\documentclass[11pt,a4paper]{article}

\usepackage[utf8]{inputenc}
\usepackage{amsmath}
\usepackage{amsfonts}
\usepackage[colorlinks,linkcolor=blue,anchorcolor=blue,citecolor=blue]{hyperref}
\usepackage{amssymb}
\usepackage{extarrows}
\usepackage{algorithmic}
\usepackage{algorithm}
\usepackage{caption}
\usepackage{amsthm}
\usepackage{comment}
\usepackage{bm}
\usepackage{tikz}
\usetikzlibrary{arrows.meta}       
\usetikzlibrary{positioning}       
\usetikzlibrary{calc}              
\usetikzlibrary{shapes}  
\usepackage{verbatim}
\usepackage{hyperref}
\usepackage{cite}
\usepackage[numbers]{natbib}
\usepackage{enumerate}
\PassOptionsToPackage{normalem}{ulem}
\usepackage{ulem}
\usepackage[margin=1in]{geometry} 
\usepackage{array}
\usepackage{cases}
\usepackage{mathrsfs}
\usepackage{longtable}
\allowdisplaybreaks[4]
\allowdisplaybreaks
\newcolumntype{L}[1]{>{\raggedright\let\newline\\\arraybackslash\hspace{0pt}}m{#1}}
\newcolumntype{C}[1]{>{\centering\let\newline\\\arraybackslash\hspace{0pt}}m{#1}}
\newcolumntype{R}[1]{>{\raggedleft\let\newline\\\arraybackslash\hspace{0pt}}m{#1}}

\usepackage{color}
\usepackage{dsfont}
\usepackage{marginnote}
\usepackage{enumitem}
\usepackage{graphicx}
\usepackage{comment}
\usepackage{float}
\usepackage{subfigure}

\theoremstyle{plain}
\newtheorem{theorem}{\protect Theorem}[section]
\newtheorem{prop}[theorem]{\protect Proposition}

\newtheorem{lemma}[theorem]{\protect Lemma}
\newtheorem{remark}[theorem]{\protect Remark}
\newtheorem{ass}{\protect Assumption}

\def\d{\mathrm{d}}

\def\as{\mathrm{a.s.}}

\newcommand{\R}{\mathbb{R}}
\newcommand{\E}{\mathbb{E}}

\newcommand{\T}{\top}

\newcommand{\F}{\mathcal{F}}
\newcommand{\Fb}{\mathbb{F}}
\newcommand{\Pb}{\mathbb{P}}

\newcommand{\tr}{\mathrm{tr}}

\newcommand{\pa}{\partial}

\newcommand{\A}{\mathscr{A}}

\title{Model-free Reinforcement Learning for Continuous Time and State: A Stochastic Maximum Principle Approach}

\author{Lijun Bo\thanks{Email: lijunbo@ustc.edu.cn, School of Mathematics and Statistics, Xidian University, Xi'an, 710126, China.}
\and
Yijie Huang\thanks{Email: yijie.huang@polyu.edu.hk, Department of Applied Mathematics, The Hong Kong Polytechnic University, Kowloon, Hong Kong, China.}
\and
Jingfei Wang \thanks{Email: wjf2104296@mail.ustc.edu.cn, School of Mathematical Sciences, University of Science and Technology of China, Hefei, 230026, China.}
}
\date{\vspace{-0.3cm}}

\begin{document}
	\maketitle
	\begin{abstract}
This paper develops a model-free reinforcement learning (RL) algorithm based on the stochastic maximum principle for continuous-time stochastic control problems with continuous state and action spaces. For a parameterized Markovian policy, we establish the existence of the decoupling field for the adjoint backward stochastic differential equation, which allows the Hamiltonian gradient to be represented as a deterministic function of time and state. We then learn this Hamiltonian gradient directly from data and incorporate it into a policy gradient scheme with inexact gradients. 
We establish the convergence of the resulting policy gradient algorithm and establish the corresponding error estimates. Under suitable conditions on learning rates, exploration parameters, and approximation errors, the objective values converge and the $L^2$-norm of the policy gradient vanishes asymptotically. We further show that the proposed SMP-based policy gradient representation is equivalent to the existing continuous-time deterministic policy gradient representation based on the advantage-rate function. The effectiveness and efficiency of the proposed RL algorithm are demonstrated by numerical experiments.
        
	\end{abstract}
	
	\section{Introduction}
    Reinforcement learning (RL) provides a data-driven framework for solving sequential decision-making problems without requiring an explicit model of the underlying environment. In this paper, we consider model-free RL for stochastic control in continuous time and continuous state and action spaces. More precisely, the model coefficients of the controlled stochastic system are assumed to be unknown, and the objective is
    to learn an optimal policy directly from trajectory data.
    The main challenge is therefore to construct an effective policy-gradient
    method that can be implemented without knowledge of the underlying model.
    
    A natural approach to this problem is to optimize a parameterized policy
    $\alpha^\psi$ by gradient descent. Existing continuous-time RL
    methods mainly build upon the dynamic programming principle (DPP) and the
    associated Hamilton-Jacobi-Bellman (HJB) framework. In contrast, the
    stochastic maximum principle (SMP) provides an alternative, variational
    approach to characterizing the policy gradient. In particular, through the
    adjoint backward stochastic differential equation (BSDE), the SMP yields a
    representation of the gradient $\nabla_\psi J(\alpha^\psi)$ in terms of the gradient of
    the Hamiltonian with respect to the action.  This representation involves
    an adjoint BSDE, which links
    the policy gradient to the state and adjoint processes along the controlled
    trajectory.
    
	\subsection{Literature Review}
	
The SMP method is one of the two fundamental approaches to solve stochastic optimal control problems, together with the HJB approach. The SMP characterizes optimal controls through a system of adjoint equations, thereby providing necessary and sufficient optimality conditions for stochastic control problems. The duality-based formulation of the SMP was developed systematically in the early literature, notably by Bismut~\cite{Bismut1978}, and was subsequently extended to nonlinear stochastic control problems. In particular, Peng~\cite{Peng1990} established a general stochastic maximum principle for nonlinear controlled diffusions with possibly nonconvex control domains. A central feature of the SMP is that the associated adjoint processes are characterized by BSDE, while the optimality condition can be expressed in terms of the Hamiltonian. This approach has been extensively developed for a broad range of stochastic control problems, including partially observed systems~\cite{Haussmann1987,Tang1998}, stochastic systems with delays~\cite{ChenWu2010,Yu2012}, time-inconsistent problem~\cite{Huang,Alia2019,MT2023}, mean-field control (MFC) problems~\cite{Carmona,CarmonaSMP,BWWY}  and extended MFC under general dynamic constraints~\cite{BWY}.
	
Beyond these theoretical extensions, the SMP offers several advantages
over the HJB approach. While the HJB approach typically relies heavily on the validity of the DPP, the SMP provides optimality conditions directly through the state and adjoint processes. This makes the SMP particularly attractive for stochastic control problems with structural features for which the DPP approach may not be applicable, such as time-inconsistent problems and stochastic control problems with general dynamic constraints. Furthermore, under suitable regularity conditions, the gradient of the objective with respect to the control can be represented in terms of the gradient of the Hamiltonian with respect to the control variable. This gradient-based characterization makes the SMP naturally compatible with gradient-based algorithms, providing a natural foundation for developing data-driven approaches to continuous-time stochastic control problems.
	
In parallel, RL has emerged as a data-driven framework for solving sequential decision-making problems without explicit knowledge of the underlying system model. While most classical RL algorithms have been developed in discrete-time  settings, many physical and financial systems evolve continuously in  time, which has motivated a rapidly growing literature on  continuous-time RL theory and algorithms.  To facilitate exploration in continuous-time learning problems, Wang et  al. \cite{WZZ2020} pioneered the incorporation of entropy regularization into continuous-time diffusion models, providing a principled formulation of  stochastic policies. Building upon this idea, Jia and Zhou \cite{JZ2022,JZ2022b,Zhou} established a theoretical framework for model-free continuous-time  RL with stochastic policies by developing policy evaluation, policy gradient, and $q$-learning algorithms through martingale characterizations. Following this line of research, substantial progress has been made in  the theoretical analysis and applications of continuous-time RL with  stochastic policies; see, e.g., \cite{BHY2025,BHYZ2026,GLZ2026,HWZ2023,HWZ2025,TZZ22}.   Related  extensions have also been developed for continuous-time mean-field control problems and mean-filed games \cite{ren26a,ren26b,weiyu2025}, optimal stopping  problems \cite{DSXZ2026,DFX2024,D2024}  as well as optimal switching and impulse control problems \cite{CDY2025,DDL2025,DPW2025,HMYZ2025}.
	
A large body of existing continuous-time RL frameworks relies on  stochastic policies, which are often formulated through entropy regularization.  As pointed out by Cheng et al.~\cite{GuoDPG}, although stochastic policies (or relaxed controls) provide a theoretically elegant framework, their implementation in continuous-time environments with continuous state and action spaces can be challenging due to the difficulty of sampling relaxed controls. Moreover, Jia et al. \cite{JOZ2026}  highlighted potential measurability issues arising from continuously sampling actions from stochastic policies.  In contrast, deterministic policy gradient (DPG) methods offer an  attractive alternative in continuous action spaces.  In discrete-time RL, Silver et al. \cite{DLHDWR2014}  showed that deterministic policies can significantly  simplify policy optimization, since the policy gradient depends only on  the state distribution rather than the full state-action distribution. Motivated by these advantages, recent studies have extended deterministic policy gradient methods to continuous-time stochastic control problems.  Cheng et al.~\cite{GuoDPG} derived a continuous-time deterministic policy gradient representation based on the advantage-rate function and developed a model-free algorithm via martingale characterizations. This framework has subsequently been extended to more complex control settings, including time-inconsistent stochastic control problems \cite{GHY2026} and mean-field control problems \cite{CGPZ2026}.
	
	A fundamental challenge in DPG methods is that the actor update requires accurate gradients of the critic with respect to the action variable, rather than merely accurate value estimates. In conventional actor--critic methods, the critic is typically trained  at the level of function values, which does not necessarily guarantee accurate gradients required for policy improvement.  To address this issue, Silver et al.~\cite{DLHDWR2014} introduced compatible function approximation to ensure consistency between critic approximation and the policy gradient structure. More recently, D'Oro and Ja{\'s}kowski \cite{PW2020} demonstrated the importance of directly learning action gradients of the critic and proposed a gradient-aware critic learning approach.  These developments highlight the importance of accurately learning gradient information in policy optimization.
	
Most existing continuous-time RL methods, including DPG-based approaches, are built upon the DPP framework, while SMP-based RL remains relatively unexplored. Archibald et al.~\cite{YongSMP} developed an SMP-based RL method that combines online estimation of a parameterized environment with backward action learning for policy improvement. However, their approach requires a parametric representation of the unknown environment in order to apply the maximal likelihood estimate. This motivates the development of fully model-free SMP-based reinforcement learning methods that directly exploit the gradient structure provided by the SMP.

	\subsection{Our Contributions}
	
	The first contribution of this paper is to develop, to the best of our knowledge, the first data-driven SMP-based algorithm for model-free continuous-time reinforcement learning. Existing SMP-based RL methods, such as Archibald et al.~\cite{YongSMP}, consider a model-based setting, where the coefficients of the controlled system and the cost functional are parameterized by $b^{\theta}$, $\sigma^{\theta}$, $f^{\theta}$ and $g^{\theta}$. They first estimate the model parameter $\theta$ from data via maximum likelihood and then exploit the Hamiltonian representation provided by the SMP to perform policy gradient. In contrast, our approach does not require a parametric model of either the system dynamics or the cost functions. Instead, both the policy and the Hamiltonian gradient information required for policy optimization are learned directly from   data, leading to a fully data-driven and model-free framework for continuous-time reinforcement learning.
	
	The second contribution is a novel policy gradient framework based on the direct learning of the Hamiltonian gradient via the SMP. In contrast to the DPG approach of Cheng et	al.~\cite{GuoDPG}, where the advantage function is first learned and then differentiated to obtain the policy gradient, we directly learn
	the Hamiltonian gradient and use it for policy optimization. To this end, we first establish a decoupling field for the adjoint BSDE in the SMP (see Lemma~\ref{lem:decoupling_field}). This allows us to represent 	the Hamiltonian gradient, for any given policy $\alpha^{\psi}$ as a deterministic function $H_a^{\alpha^{\psi}}(t,x)$ of $(t,x)$ (see
	\eqref{Hamitonian_gradient}). We then learn this function by minimizing	the $L^2$ loss in \eqref{loss_func}, and incorporate the approximation in \eqref{value_approximation} into a gradient descent scheme with inexact gradients.
	
A key advantage of directly learning the Hamiltonian gradient is that the approximation error can be directly quantified. In contrast, the approach of Cheng et al.~\cite{GuoDPG} first approximates the advantage function and then
differentiates the learned function. In general, convergence of a sequence of approximating functions does not imply convergence of their gradients, so the approximation error of the advantage function alone does not provide a direct upper bound for the error between the true policy gradient and the gradient obtained from the learned function. By learning the Hamiltonian gradient itself, our approach avoids this additional differentiation step and yields both a convergence result and a direct error estimate for the learned gradient (see Proposition~\ref{prop:inner_loop}). Moreover, we establish the convergence of the overall algorithm together with the associated error analysis  (see Theorem~\ref{thm:algorithm_convergence}), which is not provided in Cheng et al.~\cite{GuoDPG}. The price paid is a higher exploration cost. Specifically, learning $H_a^{\alpha^{\psi}}(t,x)$ with a parameterized function $h^{\theta}(t,x)$ for $\theta\in\mathbb{R}^q$, entails $q$ explorations, which is a natural part of the parameterized learning formulation using SMP. Finally, although the resulting representation of the policy gradient
$\nabla_{\psi}\mathcal{J}(\psi)$ in \eqref{J_gradient} is formally different from that of Cheng et al.~\cite{GuoDPG}, we establish their equivalence in Lemma \ref{lem:equivalence-dpg}. The underlying reason for this difference is that the spatial gradient of the value function $V^{\psi}(t,x)$ in Cheng et al.~\cite{GuoDPG} does not in general coincide with the solution of the adjoint BSDE in SMP.
	
The third contribution concerns the information required by the learning algorithm. Compared with existing continuous-time RL methods, such as Jia and Zhou~\cite{Zhou} and Cheng et al.~\cite{GuoDPG}, our framework requires less information about the cost structure. In their approaches, the agent needs to observe the running cost at any given time along each realized trajectory, as well as the terminal cost at the end of the
trajectory. In contrast, our algorithm only requires access to the total cost $J^{\alpha}=\int_0^T f(t,X_t^{\alpha},\alpha_t)\d t
+g(X_T^{\alpha})$ associated with a realized trajectory, without requiring the running cost to be observed at individual time points. Consequently, the proposed framework can be applied to settings where the running and terminal costs are not separately observable and only the total performance of a realized trajectory is available.
This reduction in the required cost information, however, comes with a trade-off. Since the total cost $J^{\alpha}$ is only available after the entire trajectory has been realized, the proposed method is inherently episodic and cannot directly perform online learning or update the policy before the end of an episode.
	
	The remainder of the paper is organized as follows: Section~\ref{sec:problem} formulates the model-free reinforcement learning problem in continuous time and state and introduces the parameterized policy family. Section~\ref{sec:SMP_RL} presents the proposed SMP-based reinforcement learning framework, including learning the Hamiltonian gradient $H_a^{\alpha^{\psi}}$ for a given parameterized policy $\alpha^{\psi}$, the subsequent policy parameter update and the overall algorithm procedure summarized in Algorithm~\ref{main_algorithm}. Section~\ref{sec:main_result} presents the main theoretical results: (i) the convergence and error analysis for the inner loop that learns the Hamiltonian gradient $H_a^{\alpha^{\psi}}$, (ii) the convergence and error analysis for the overall policy gradient algorithm. Section~\ref{sec:numerical} applies the proposed algorithm to two numerical examples and and demonstrates its effectiveness and efficiency. Section \ref{sec:conclusion} summarizes the main results and discusses some future research directions. Finally, Appendix~\ref{sec:proof_main} and Appendix~\ref{sec:proof_auxi} collect the proofs of the main results and auxiliary results in previous sections respectively.

{\bf Notations.}\quad We list below some notations that will be used frequently throughout the paper:
	\begin{center}
		\begin{longtable}{l l}
			$|\cdot|$ & Euclidean norm on Euclidean spaces\\
			$\nabla~(\nabla^2)$ & The (full) gradient (Hessian) operator\\
			$\nabla_x~(\nabla_{x}^2)$ & The (partial) gradient (Hessian) operator with respect to $x$\\
			$L^p((A,\mathscr{B}(A),\lambda_A);E)$ & Set of $L^p$-integrable $E$-valued mapping defined on $(A,\mathscr{B}(A))$\\ 
			& and we write $L^p(A;E)$ for short\\
			$\|\cdot\|_{L^p(A;E)}$ & The $L^p$ norm in $L^p(A;E)$ and we write $\|\cdot\|_{L^p}$ for short\\
			& when the underlying space is clear from the context.\\
			${\tt m}$ & The Lebesgue measure over $[0,T]$.\\
		\end{longtable}
	\end{center} 
	
	\section{Problem Formulation}\label{sec:problem}
	In this section, we formulate the reinforcement learning (RL) problem with continuous time and space in a model-free setting. The objective of the agent is to learn the optimal parameterized Markovian policy for a continuous-time stochastic control problem with unknown coefficients.
	
	Let $T>0$ be a finite horizon and $d,r,l\in\mathbb{N}$. Consider a complete probability space $(\Omega,\F,\Pb)$ which supports a standard $r$-dimensional Brownian motion $W=(W_t)_{t\in [0,T]}$ and a $\R^d$-valued random variable $\xi$ with $\E[|\xi|^4]<\infty$. Introduce the filtration $\Fb=(\F_t)_{t\in [0,T]}$ defined by $\F_t=\sigma(\xi,B_s;s\leq t)$ for $t\in [0,T]$. Without loss of generality (WLOG), we may assume that this filtration $\Fb$ satisfies the usual conditions. Let the policy space be $\R^l$, and denote by $\A[0,T]$ the collection of all $\R^l$-valued $\Fb$-progressively measurable processes $\alpha=(\alpha_t)_{t\in [0,T]}$ satisfying $\E[\int_0^T|\alpha_t|^2\d t]<\infty$. For any $\alpha\in\A[0,T]$, let $X^{\alpha}=(X_t^{\alpha})_{t\in [0,T]}$ evolves as the following controlled stochastic differential equation (SDE):
	\begin{align}\label{SDE_dynamic}
		X_t^{\alpha}=\xi+\int_0^tb(s,X_s^{\alpha},\alpha_s)\d s+\int_0^t\sigma(s,X_s^{\alpha},\alpha_s)\d W_s,\quad \forall t\in [0,T],
	\end{align}
	where, $(b,\sigma):[0,T]\times\R^d\times\R^l\to\R^d\times\R^{d\times r}$ are Borel measurable functions so that Eq.~\eqref{SDE_dynamic} admits a unique (strong) solution. The goal is to minimize the following cost functional over all $\alpha\in\A[0,T]$:
	\begin{align}\label{cost_func}
		J(\alpha):=\E\left[\int_0^Tf(t,X_t^{\alpha},\alpha_t)\d t+g(X_T^{\alpha})\right],
	\end{align}
	where the running cost $f:[0,T]\times\R^d\times\R^l\to\R$ and the terminal cost $g:\R^d\to\R$ are measurable mappings.
	
It is well-known that, under mild conditions (c.f. Haussmann and Lepeltier \cite{Haussmann1990}), the optimal policy $\alpha^*=(\alpha_t^*)_{t\in [0,T]}$ admits a Markovian (feedback) representation. Specifically,  there exists a Borel measurable mapping $\phi:[0,T]\times\R^d\to \R^l$ such that $\alpha^*_t=\phi(t,X_t^*)$, where $X^*=(X_t^*)_{t\in [0,T]}$ is the unique solution to the following SDE: 
\begin{align*}
X_t^*=\xi+\int_0^t b(s,X_s^*,\phi(s,X_s^*))\d s +\int_0^t\sigma(s,X_s^*,\phi(s,X_s^*))\d W_s,\quad \forall t\in[0,T].
\end{align*}
Hence, throughout the paper, we will assume that the optimal policy of control problem \eqref{SDE_dynamic}-\eqref{cost_func} admits the Markovian (feedback) form. 
	
In the model-free RL framework, the agent has {\it no access} to the model coefficients $(b,\sigma,f,g)$. Instead, the agent interacts with the environment \eqref{SDE_dynamic} by exploring different policies and improves her policy through the observed states, actions and costs. More precisely, for any admissible control $\alpha\in\A[0,T]$ and a given time discretization given by $0=t_0<t_1<\cdots<t_{N_T}=T$, the agent can only observe $N$ independent realizations of the state-control pairs and the total cost given by
	\begin{align*}
		\left\{\left(X_{t_j}^{\alpha},\alpha_{t_j},J^{\alpha}=\sum_{\ell=0}^{N_T-1}f(t_\ell,X_{t_\ell}^{\alpha},\alpha_{t_\ell})(t_{\ell+1}-t_\ell)+g(X_{t_{N_T}}^{\alpha})\right)\right\}_{j=0}^{N_T}.
	\end{align*}
Correspondingly, we denote these samples by $\{(X_{t_j,i}^{\alpha},\alpha_{t_j,i}, J^{\alpha}_i)\}_{i=1}^{N}$ for $j=0,1,\ldots,N_T$. In order to solve the RL problem for control problem \eqref{SDE_dynamic}-\eqref{cost_func}, we approximate the optimal feedback mapping $\phi:[0,T]\times\mathbb{R}^d\to \R^l$ by a parameterized class of functions described as follows:
	\begin{align}\label{parameterized_policy}
		\left\{u^{\psi}:[0,T]\times\mathbb{R}^d\to \R^l;~\psi\in\mathbb{R}^p\right\},
	\end{align}
where $p\geq1$ is the dimension of parameter space. Lastly, we impose the following assumptions on model coefficients and the parameterized policies throughout the paper.
	\begin{ass}\label{ass1}
		\begin{itemize}
			\item [{\rm (A1)}] The functions $(b,\sigma):[0,T]\times\R^d\times\R^l\to \R^d\times\R^{d\times r}$ are joitly continuous mappings and are twice continuously differentiable in $(x,a)\in\R^d\times\R^l$. Moreover, there exists a constant $M>0$ such that $|\nabla_{(x,a)} (b,\sigma)(t,x,a)|+|\nabla_{(x,a)}^2 (b,\sigma)(t,x,a)|\leq M$ for all $(t,x,a)\in [0,T]\times\R^d\times\R^l$.
			
\item [{\rm (A2)}] The functions $f:[0,T]\times\R^n\times \R^l\to\R$ and $g:\R^n\to\R$ are jointly continuous mappings and are twice continuously differentiable in $(x,a)\in\R^d\times\R^l$. Moreover, there exists a constant $M>0$ such that $|\nabla_{(x,a)}^2f(t,x,a)|+|\nabla_x^2g(x)|\leq M$ for all $(t,x,a)\in [0,T]\times\R^d\times \R^l$.
			
\item [{\rm (A3)}]  For any $\psi\in\R^p$, $u^{\psi}$ is jointly continuous and is twice continuously differentiable in $x\in \R^d$. Moreover, there exists a mapping $M_1:[0,T]\to(0,\infty)$ satisfying $M:=\int_0^T|M_1(t)|^2\d t<\infty$ such that $|\nabla_x u^{\psi}(t,x)|+|\nabla_{x}^2u^{\psi}(t,x)|\leq M_1(t)$ for all $(\psi,t,x)\in\R^p\times [0,T]\times\R^d$. Furthermore, assume that $u^{\psi}(t,x)$ is also continuously differentiable in $\psi\in\R^p$ for every $(t,x)\in [0,T]\times\R^d$, and there exists a constant $M>0$ such that, for all $(t,x)\in [0,T]\times\R^d$ and $\psi^1,\psi^2\in\R^p$, 
			\begin{align*}
				\left|u^{\psi^1}(t,x)-u^{\psi^2}(t,x)\right|&+\left|\nabla_xu^{\psi^1}(t,x)-\nabla_xu^{\psi^2}(t,x)\right|\nonumber\\
				&\quad+\left|\nabla_{\psi}u^{\psi^1}(t,x)-\nabla_{\psi}u^{\psi^2}(t,x)\right|\leq M(1+|x|)\left|\psi^1-\psi^2\right|.
			\end{align*}
		\end{itemize}
	\end{ass}
	In particular, the collection of the parameterized polices $u^{\psi}$ satisfying Assumption (A3) above is denoted by $\Psi$. Then, thanks to Assumption~\ref{ass1}, we can deduce that, for any parameterized policy $u^{\psi}\in\Psi$, the following controlled SDE has a unique (strong) solution $X^{\alpha^{\psi}}=(X_t^{\alpha^{\psi}})_{t\in[0,T]}$:
	\begin{align}\label{SDE_parameter}
		X_t^{\alpha^{\psi}}=\xi+\int_0^t b(s,X_s^{\alpha^{\psi}},
		u^{\psi}(s,X_s^{\alpha^{\psi}}))\d s+\int_0^t \sigma(s,X_s^{\alpha^{\psi}},u^{\psi}(s,X_s^{\alpha^{\psi}}))\d W_s,
	\end{align}
	where, $\alpha^{\psi}=(\alpha_t^{\psi})_{t\in[0,T]}$ denotes the feedback control induced by $u^{\psi}$ in the sense that $\alpha_t^{\psi}=u^{\psi}(t,X_t^{\alpha^{\psi}})$ for $t\in[0,T]$. The following lemma provides a priori moment estimates for $X^{\alpha^{\psi}}$.
\begin{lemma}\label{lem:X_moment}
Let Assumption~\ref{ass1} hold. There exists a constant $C>0$ depending only on $T$ and $M$ introduced in Assumption~\ref{ass1} such that $\E[\sup_{t\in[0,T]}|X_t^{\alpha^{\psi}}|^4]\leq C(1+\E\left[|\xi|^4\right])$ for all $\psi\in\R^p$. 
\end{lemma}
	
We are thus led to consider the following parameterized optimization problem described as:
	\begin{align}\label{cost_func_parameter}
		\mathcal{J}(\psi):=\E\left[\int_0^T f(t,X_t^{\alpha^{\psi}},\alpha_t^{\psi})\d t+g(X_T^{\alpha^{\psi}})\right]\to\inf_{\psi\in\R^p}\mathcal{J}(\psi).
	\end{align}
	In view of Lemma~\ref{lem:X_moment} and Assumption~\ref{ass1}-(A2), the optimization problem \eqref{cost_func_parameter} is well-posed. Moreover, it holds that
	\begin{align}\label{lower_bound}
		\inf_{\psi\in\R^p}\mathcal{J}(\psi)>-\infty.
	\end{align}
	
	\section{SMP-based Reinforcement Learning}\label{sec:SMP_RL}
	In this section, we introduce our RL algorithm based on the stochastic maximum principle (SMP). To this end, we use $\langle\cdot,\cdot\rangle_{L^2}$ to denote the $L^2$-inner product on $L^2([0,T]\times\Omega;\R^l)$, i.e., $\langle \alpha,\beta\rangle_{L^2}:=\E[\int_0^T\alpha_t^{\T}\beta_t\d t]$ for any $\alpha,\beta\in L^2([0,T]\times\Omega;\R^l)$. Furthermore, denote by $\|\cdot\|_{L^2}$ the norm induced by $\langle\cdot,\cdot\rangle_{L^2}$.
	
	\subsection{The Hamiltonian Gradient}
	
	The Hamiltonian $H:[0,T]\times\R^d\times\R^l\times\R^d\times\R^{d\times r}\to\R$ corresponding to control problem \eqref{SDE_dynamic}-\eqref{cost_func} is given by
	\begin{align}\label{Hamiltonian}
		H(t,x,a,y,z)=b(t,x,a)^{\T}y+\tr\left(\sigma(t,x,a)^{\T}z\right)+f(t,x,a).
	\end{align}
	For any admissible control strategy $\alpha\in\A[0,T]$, let $(Y^{\alpha},Z^{\alpha})=(Y_t^{\alpha},Z_t^{\alpha})_{t\in [0,T]}$ be the unique $\R^d\times\R^{d\times r}$-valued solution to the following BSDE given by
	\begin{align}\label{eq:BSDE}
		\d Y_t^{\alpha}=-\nabla_xH(t,X_t^{\alpha},\alpha_t,Y_t^{\alpha},Z_t^{\alpha})\d t+Z_t^{\alpha}\d W_t,\quad Y_T^{\alpha}=\nabla_xg(X_T^{\alpha}),
	\end{align} 
	where, $X^{\alpha}=(X_t^{\alpha})_{t\in [0,T]}$ follows the controlled dynamics described as \eqref{SDE_dynamic}. In light of the classical SMP (c.f. Peng \cite{Peng1990} and Lemma 3.1 in Acciaio et al. \cite{CarmonaSMP}), we have, for any $\alpha,\beta\in\A[0,T]$, \begin{align}\label{value_approximation}
		\frac{J(\alpha+\epsilon\beta)-J(\alpha)}{\epsilon}=\E\left[\int_0^T\nabla_aH(t,X_t^{\alpha},\alpha_t,Y_t^{\alpha},Z_t^{\alpha})^{\T}\beta_t\d t\right]+O(\epsilon\|\beta\|_{L^2}^2),
	\end{align}
as $\epsilon\to 0$. Here, by convention, the notation $O(\epsilon\|\beta\|_{L^2}^2)$ means that there exists a constant $C>0$ independent of $\epsilon$ and $\beta$ such that $O(\epsilon\|\beta\|_{L^2}^2)\leq C\epsilon\|\beta\|_{L^2}^2$. As a result, we can compute the gradient of the cost functional $\mathcal{J}(\psi)$ with respect to $\psi\in\R^p$ by following directly from \eqref{value_approximation} and the chain rule.
\begin{lemma}\label{lem:J_gradient}
Let Assumption~\ref{ass1} hold. The parameterized function $\mathcal{J}(\psi)$ defined in \eqref{cost_func_parameter} is differentiable with respect to $\psi\in\R^p$. Moreover, it holds that
\begin{align}\label{J_gradient}
&\nabla_{\psi}\mathcal{J}(\psi)\nonumber\\
&\quad=\E\left[\int_0^T\left(\nabla_{\psi}u^{\psi}(t,X_t^{\alpha^{\psi}})+\nabla_xu^{\psi}(t,X_t^{\alpha^{\psi}})\nabla_{\psi}X_t^{\alpha^{\psi}}\right)^{\T}\nabla_aH(t,X_t^{\alpha^{\psi}},\alpha_t^{\psi},Y_t^{\alpha^{\psi}},Z_t^{\alpha^{\psi}})\d t\right].
\end{align}
Here, the state-control pair $(X^{\alpha^{\psi}},\alpha^{\psi})$ is introduced in \eqref{SDE_parameter} and the gradient process $\nabla_{\psi}X^{\alpha^{\psi}}=(\nabla_{\psi}X_t^{\alpha^{\psi}})_{t\in [0,T]}$ is the unique solution to the following SDE, for $t\in[0,T]$,
		\begin{align}\label{X_gradient}
			\begin{cases}
				\displaystyle d(\nabla_{\psi}X^{\alpha^{\psi}}_t)=\left(\nabla_xb(t,X_t^{\alpha^{\psi}},\alpha_t^{\psi})+\nabla_ab(t,X_t^{\alpha^{\psi}},\alpha_t^{\psi})\nabla_xu^{\psi}(t,X_t^{\alpha^{\psi}})\right)\nabla_{\psi}X_t^{\alpha^{\psi}}\d t\\[1em]
\displaystyle\qquad\quad+\left(\nabla_x\sigma(t,X_t^{\alpha^{\psi}},\alpha_t^{\psi})+\nabla_a\sigma(t,X_t^{\alpha^{\psi}},\alpha_t^{\psi})\nabla_xu^{\psi}(t,X_t^{\alpha^{\psi}})\right)\nabla_{\psi}X_t^{\alpha^{\psi}}\d W_t\\[1em]
\displaystyle\qquad\quad+\nabla_ab(t,X_t^{\alpha^{\psi}},\alpha_t^{\psi})\nabla_{\psi}u^{\psi}(t,X_t^{\alpha^{\psi}})\d t+\nabla_a\sigma(t,X_t^{\alpha^{\psi}},\alpha_t^{\psi})\nabla_{\psi}u^{\psi}(t,X_t^{\alpha^{\psi}})\d W_t,\\[1em]
\displaystyle\nabla_{\psi}X^{\alpha^{\psi}}_0=\boldsymbol{0}_{d\times p}.
\end{cases}
\end{align}
\end{lemma}
By applying the SMP (cf. Theorem~3.12 in Bo et al.~\cite{BWY}), a necessary condition for the optimal policy parameter $\psi^*$ is given by $\nabla_a H(t,X_t^{\alpha^{\psi^*}},\alpha_t^{\psi^*},Y_t^{\alpha^{\psi^*}},Z_t^{\alpha^{\psi^*}})=0$, ${\tt m}\otimes\Pb\text{-}$a.s.. Consequently, it follows from Lemma \ref{lem:J_gradient} that $\nabla_{\psi}\mathcal{J}(\psi^*)=0$. 

We now fix an arbitrary policy $u^{\psi}\in\Psi$. The next lemma states that under the feedback policy $\alpha^{\psi}=(\alpha_t^{\psi})_{t\in[0,T]}$, the BSDE solution $(Y^{\alpha^{\psi}},Z^{\alpha^{\psi}})=(Y_t^{\alpha^{\psi}},Z_t^{\alpha^{\psi}})_{t\in[0,T]}$ admits a decoupling field, whose proof is delegated to the Appendix~\ref{sec:proof_auxi}.
\begin{lemma}\label{lem:decoupling_field}
Let Assumption~\ref{ass1} hold. Consider the following forward and backward stochastic differential equation (FBSDE) given by 
\begin{align}\label{eq:FBSDES}
\begin{cases}
\displaystyle\d X_t^{\alpha^{\psi}}=b(t,X_t^{\alpha^{\psi}},u^{\psi}(t,X_t^{\alpha^{\psi}}))\d t+\sigma(t,X_t^{\alpha^{\psi}},u^{\psi}(t,X_t^{\alpha^{\psi}}))\d W_t,\quad X_0^{\alpha^{\psi}}=\xi,\\[1em]
\displaystyle\d Y_t^{\alpha^{\psi}}=-\nabla_xH(t,X_t^{\alpha^{\psi}},u^{\psi}(t,X_t^{\alpha^{\psi}}),Y_t^{\alpha^{\psi}},Z_t^{\alpha^{\psi}})\d t+Z_t^{\alpha^{\psi}}\d W_t,\quad Y_T^{\alpha^{\psi}}=\nabla_xg(X_T^{\alpha^{\psi}}).
\end{cases} 
\end{align}
Then, FBSDE \eqref{eq:FBSDES} has a jointly continuous decoupling field $v^{\psi}:[0,T]\times \R^d\to\R^d$ that is continuously differentiable with respect to the spatial variable $x\in\R^d$. More precisely, it holds that, ${\tt m}\otimes \Pb$-a.s. 
\begin{align}\label{decoupling_field}
Y_t^{\alpha^{\psi}}=v^{\psi}(t,X_t^{\alpha^{\psi}}),\quad Z_t^{\alpha^{\psi}}=\nabla_x v^{\psi}(t,X_t^{\alpha^{\psi}})\sigma(t,X_t^{\alpha^{\psi}},u^{\psi}(t,X_t^{\alpha^{\psi}})),\quad \forall t\in[0,T].
\end{align}
Furthermore, there exists a constant $C>0$ depending on $T$ and $M$ introduced in Assumption~\ref{ass1} only such that
\begin{align}\label{gradient_bound}
\sup_{(t,x)\in [0,T]\times\R^d}\left|\nabla_xv^{\psi}(t,x)\right|\leq C(1+|x|).
\end{align}
\end{lemma}
	
As presented in Lemma~\ref{lem:J_gradient}, the computation of the gradient $\nabla_{\psi}\mathcal{J}(\psi)$ will boil down to evaluating the Hamiltonian gradient $\nabla_aH(t,X_t^{\alpha^{\psi}},\alpha_t^{\psi},Y_t^{\alpha^{\psi}},Z_t^{\alpha^{\psi}})$. Hence, by utilizing the decoupling field $v^{\psi}$ obtained in Lemma~\ref{lem:decoupling_field}, we may define the following mapping $H^{\alpha^{\psi}}_a:[0,T]\times\R^d\to\R^l$ by, for any $(t,x)\in [0,T]\times\R^d$,
	\begin{align}\label{Hamitonian_gradient}
		H^{\alpha^{\psi}}_a(t,x):=\nabla_aH\left(t,x,u^{\psi}(t,x),v^{\psi}(t,x),\nabla_x v^{\psi}(t,x)\sigma\left(t,x,u^{\psi}(t,x)\right)\right).
	\end{align}
	Thus, we have $H_a^{\alpha^{\psi}}(t,X_t^{\alpha^{\psi}})=\nabla_aH(t,X_t^{\alpha^{\psi}},\alpha_t^{\psi},Y_t^{\alpha^{\psi}},Z_t^{\alpha^{\psi}})$ for $t\in[0,T]$. Consequently, the agent only needs to learn the mapping $H_a^{\alpha^{\psi}}$, evaluate it along the state trajectory $(t,X_t^{\alpha^{\psi}})_{t\in [0,T]}$ and perform gradient descent according to the gradient \eqref{J_gradient}.
	
\subsection{Exploration: Learning $H_a^{\alpha^{\psi}}(t,x)$}
	
In order to approximate the mapping $H_a^{\alpha^{\psi}}$ given by \eqref{Hamitonian_gradient}, we consider a family of parametrized mappings which is given by 
\begin{align}\label{parameterized_policynablazH}
\left\{h^{\theta}:[0,T]\times\R^d\to\R^l;~\theta\in\R^q\right\},   
\end{align}
where, $q\geq1$ is the dimension of parameter space. We then impose the following assumption on the paramterized family.
\begin{ass}\label{ass2}
For any $(t,x)\in[0,T]\times\R^d$, $\theta\to h^{\theta}(t,x)$ is twice continuously differentiable. There exists a positive mapping $t\to M_1(t)$ satisfying $M:=\int_0^TM_1(t)^2\d t<\infty$ such that, for any $(\theta,t,x)\in\R^q\times[0,T]\times\R^d$,
\begin{align*}
\left|h^{\theta}(t,x)\right|+\left|\nabla_{\theta}h^{\theta}(t,x)\right|+\left|\nabla_{\theta}^2h^{\theta}(t,x)\right|\leq M_1(t)(1+|x|^2).
\end{align*}
\end{ass}

\begin{remark}
We conjecture that the gradient of the Hamiltonian has quadratic growth in the state variable $x\in\mathbb{R}^d$. Indeed, by Lemma~\ref{lem:decoupling_field}, the decoupling fields of the adjoint processes $(Y,Z)=(Y_t,Z_t)_{t\in[0,T]}$ admit at most quadratic growth in $x$. Consequently, $H_a^{\alpha^\psi}$ also has at most quadratic growth in $x$, owing to the uniform boundedness of $\nabla_a(b,\sigma)$ introduced in Assumption~\ref{ass1}.
\end{remark}

To learn $H_a^{\alpha^{\psi}}$ by the above family of parameterized functions, we adopt the following squared  loss function described as, for $(\theta,\psi)\in\R^q\times\R^p$,
\begin{align}\label{loss_func}
L(\theta;\psi)=\frac12\E\left[\int_0^T\left|h^{\theta}(t,X_t^{\alpha^{\psi}})-H_a^{\alpha^{\psi}}(t,X_t^{\alpha^{\psi}})\right|^2\d t\right],
\end{align}
where, $h^{\theta}:[0,T]\times\R^d\to\R^l$ is the parameterized function in \eqref{parameterized_policynablazH} satisfying Assumption \ref{ass2}. A simple calculation yields that
\begin{align}\label{loss_gradient}
\nabla_{\theta}L(\theta;\psi)&=\E\left[\int_0^T\nabla_{\theta}h^{\theta}(t,X_t^{\alpha^{\psi}})^{\T}h^{\theta}(t,X_t^{\alpha^{\psi}})\d t\right]\nonumber\\
&\quad-\E\left[\int_0^T\nabla_{\theta}h^{\theta}(t,X_t^{\alpha^{\psi}})^{\T}H^{\alpha^{\psi}}_a(t,X_t^{\alpha^{\psi}})\d t\right].
\end{align}
However, the 2nd term in \eqref{loss_gradient} can not be computed directly since it concludes $H_a^{\alpha^{\psi}}$. To overcome this difficulty, for any $\epsilon>0$, we can utilize \eqref{value_approximation} to approximate it by using 
	\begin{align*}
		\frac{J(\alpha^{\psi}+\epsilon\pa_{\theta_l}h^{\theta}(\cdot,X_{\cdot}^{\alpha^{\psi}}))-J(\alpha^{\psi})}{\epsilon},\quad \forall l\in \{1,\dots,q\}.
	\end{align*}
	Furthermore, denote by $\boldsymbol{e}^{\psi}_{\theta}=(e^{\psi}_{\theta_1},\dots,e_{\theta_q}^{\psi})^{\T}$ the corresponding error term, i.e., for any $l\in\{1,\dots,q\}$, the element $e^{\psi}_{\theta_l}$ is defined by
	\begin{align}\label{eq:errorlth}
		e^{\psi}_{\theta_l}:=\E\left[\int_0^T\partial_{\theta_l}h^{\theta}(t,X_t^{\alpha^{\psi}})^{\T}H^{\alpha^{\psi}}_a(t,X_t^{\alpha^{\psi}})\d t\right]-\frac{J(\alpha^{\psi}+\epsilon\pa_{\theta_l}h^{\theta}(\cdot,X_{\cdot}^{\alpha^{\psi}}))-J(\alpha^{\psi})}{\epsilon}.
	\end{align}
It follows from \eqref{value_approximation} that, there exists a constant $C>0$ depending on $T$ and $M$ introduced in Assumption~\ref{ass1} only such that
\begin{align}\label{error_estimate}
\left|\boldsymbol{e}^{\psi}_{\theta}\right|\leq C\epsilon\left\|\nabla_{\theta}h^{\theta}\left(\cdot,X_{\cdot}^{\alpha^{\psi}}\right)\right\|_{L^2}^2.
\end{align}
As a consequence, we can perform gradient descent with errors to update the parameter $\theta$ in the following form:
\begin{align}\label{theta_update}
\theta\leftarrow\theta-\eta\big(\nabla_{\theta}L(\theta;\psi)+\boldsymbol{e}^{\psi}_{\theta}\big),
\end{align}
where, $\eta>0$ denotes the learning rate.
	
\subsection{Exploitation: Policy Gradient}
For a fixed $\psi\in\R^p$, the iteration \eqref{theta_update} is performed until a prescribed stopping criterion is met or a maximum number of iterations is reached, and the resulting parameter is denoted by $\theta^*(\psi)$. We then proceed to learn the optimal policy $\alpha^{\psi}$ by using the policy gradient method. To do it, we first define the learning error process for $H_a^{\alpha^\psi}(t,x)$ by
	\begin{align}\label{error_process}
		I_t^\psi:=h^{\theta^*(\psi)}(t,X_t^{\alpha^\psi})-H_a^{\alpha^\psi}(t,X_t^{\alpha^\psi}),\quad \forall t\in[0,T].
	\end{align}
	Choose $\delta>0$ and define the $\R^{d\times p}$-valued finite difference process $\Delta^{\psi,\delta}=(\Delta^{\psi,\delta}_t)_{t\in [0,T]}$ by, for any $t\in[0,T]$,
	\begin{align}\label{finite_difference}
		\Delta^{\psi,\delta}_t&:=\left(\Delta_t^{\psi,\delta,1},\dots,\Delta_t^{\psi,\delta,p}\right),\quad
		\Delta_t^{\psi,\delta,m}:=\frac{X_t^{\alpha^{\psi+\delta \boldsymbol{v}_m}}-X_t^{\alpha^\psi}}{\delta},\quad \forall m=1,\dots,p.
	\end{align}
	Here, $\boldsymbol{v}_m$ denotes the $m$-th canonical basis vector of $\mathbb R^p$. 
	
	The next lemma states that the finite difference proces $\Delta^{\psi,\delta}$ is an appropriate approximation of $\nabla_{\psi}X^{\alpha^{\psi}}$, whose proof is deferred to Appendix~\ref{sec:proof_auxi}.
	\begin{lemma}\label{lem:X_gradient_approx}
		Let Assumption~\ref{ass1} hold. Recall the gradient process $\nabla_{\psi}X^{\alpha^{\psi}}=(\nabla_{\psi}X_t^{\alpha^{\psi}})_{t\in[0,T]}$ given in \eqref{X_gradient}. Define the error process $F^{\psi,\delta}=(F_t^{\psi,\delta})_{t\in [0,T]}$ by $F_t^{\psi,\delta}:=\Delta^{\psi,\delta}_t-\nabla_{\psi}X^{\alpha^{\psi}}_t$ for $t\in[0,T]$. Then, there exists a constant $C>0$ depending on $\psi$, $T$ and $M$ introduced in Assumption~\ref{ass1} only such that
		\begin{align}\label{X_gradient_approx}
			\E\left[\sup_{t\in [0,T]}\left|F_t^{\psi,\delta}\right|^2\right]\leq C\delta^2.
		\end{align}
	\end{lemma}
	
	Recall the objective functional $\mathcal{J}(\psi)$ defined by \eqref{cost_func_parameter}. Now, we can approximate $\nabla_{\psi}\mathcal{J}(\psi)$ by $P^{\psi}=(P^{\psi}_1,\dots,P^{\psi}_p)^{\T}$, where the $m$-th element is given by, for $m=1,\ldots,p$, 
	\begin{align}
		P^{\psi}_m:=\E\left[\int_0^Th^{\theta^*(\psi)}(t,X_t^{\alpha^{\psi}})^{\T}\left(\pa_{\psi_m}u^{\psi}(t,X_t^{\alpha^{\psi}})+\nabla_xu^{\psi}(t,X_t^{\alpha^{\psi}})\Delta_t^{\psi,\delta,m}\right)\d t\right].
	\end{align}
	The corresponding error is defined by
		\begin{align}\label{J_gradient_error}
     \boldsymbol{E}^{\psi}:=P^{\psi}-\nabla_{\psi}\mathcal{J}(\psi)&=\E\left[\int_0^T\left(\nabla_xu^{\psi}(t,X_t^{\alpha^{\psi}})F_t^{\psi,\delta}\right)^{\T}h^{\theta^*(\psi)}(t,X_t^{\alpha^{\psi}})\d t\right]\nonumber\\		&\quad+\E\left[\int_0^T\left(\nabla_{\psi}u^{\psi}(t,X_t^{\alpha^{\psi}})+\nabla_xu^{\psi}(t,X_t^{\alpha^{\psi}})\nabla_{\psi}X^{\alpha^{\psi}}_t\right)^{\T}I_t^{\psi}\d t\right],
	\end{align}where, $I^{\psi}=(I_t^{\psi})_{t\in[0,T]}$ is the learning error process defined by \eqref{error_process}. As a consequence, we can perform the gradient descent with errors to update the parameter $\psi$ in the following form:
	\begin{align}\label{psi_update}
		\psi\leftarrow\psi-\kappa\big(\nabla_{\psi}\mathcal{J}(\psi)+\boldsymbol{E}^{\psi}\big),
	\end{align}
	where, $\kappa>0$ denotes the learning rate.

	\subsection{Algorithm Design and Comparison}
	
	The full procedure of our SMP-based model-free policy gradient is presented in Algorithm~\ref{main_algorithm}. We next present a comparative description between our method and the most closely related known methods (c.f. Archibald et al.~\cite{YongSMP}, Jia and Zhou \cite{Zhou} and Cheng et al.~\cite{GuoDPG}).
	\begin{algorithm}[htbp]
		\begin{algorithmic}[1]
			\STATE \textbf{Input:} initial model parameter $\theta^0$, initial control parameter $\psi^0$, learning rates $\eta,\kappa$, hyperparameter $\epsilon,\delta$, max iterations $N_{iter},N_{inner}$.
			\FOR{$n=0$ \textbf{to} $N_{iter}$}
			\STATE \textbf{(1) Data collection:} 
			\STATE  Using current control $\alpha^{\psi^n}$, sample $N$ trajectories
			\begin{align*}
				\mathcal{D}^{(n)} = \left\{ \left(X^{\alpha^{\psi^n}}_{t_j,i},\alpha_{t_j,i}^{\psi^n},J^{{\alpha}^{\psi^n}}_i\right);~i=1\dots N,\; j=0\dots N_T \right\}.
			\end{align*}
			\STATE \textbf{(2) Exploration (Learning $ H^{\alpha^{\psi^n}}_{a}(t,x)$):}
			\STATE  Minimize the loss $L(\theta;\alpha^{\psi^n})$
			\FOR{$k=0$ \textbf{to} $N_{inner}$}
			\STATE  Collect cost data 
			\begin{align*}
				\mathcal{D}^{(n,k)}=\left\{J_i^{\alpha^{\psi^n}+\epsilon_{k}\pa_{\theta_l}h^{\theta^k}(\cdot,X_{\cdot}^{\alpha^{\psi^n}})};~i=1,\dots,N,\;  l=1,\dots,q\right\}.
			\end{align*} 
			\STATE Compute the gradient component  {\small\begin{align*}
					G_l^{k,n}&=\frac1N\sum_{i=1}^N\left(\sum_{j=0}^{N_T-1}h^{\theta^k}(t_j,X_{t_j,i}^{\alpha^{\psi^n}})^{\T}\pa_{\theta_l}h^{\theta^k}(t_j,X_{t_j,i}^{\alpha^{\psi^n}})(t_{j+1}-t_j)-\frac{J_i^{\alpha^{\psi^n}+\epsilon_{k}\pa_{\theta_l}h^{\theta^k}(\cdot,X_{\cdot}^{\alpha^{\psi^n}})}-J_i^{\alpha^{\psi^n}}}{\epsilon_k}\right).
			\end{align*} }
			\STATE Update the model parameter: $\displaystyle\theta^{k+1}=\theta^k-\eta_kG^{k,n}$ with $G^k=(G^{k,n}_1,\dots,G^{k,n}_q)^{\T}$.
			\ENDFOR
			\STATE  The inner-loop terminates with the model parameter $\theta^*(\psi^{n})$.
			\STATE \textbf{(3) Exploitation (Policy gradient):}
			\STATE Compute the approximated gradient process (see \eqref{finite_difference})\begin{align*}
				\left\{\Delta_{t_j,i}^{\psi^n,\delta_n,m}=\frac{X_{t_j,i}^{\alpha^{\psi^n+\delta_n \boldsymbol{v}_m}}-X_{t_j,i}^{\alpha^{\psi^n}}}{\delta_n};~i=1\dots,N,\;j=0,\dots,N_T,\;m=1,\dots,p\right\}.
			\end{align*}
			\STATE  Compute the gradient component{\small\begin{align*}
					P_m^n=\frac1N\sum_{i=1}^N\left(\sum_{j=0}^{N_T-1}h^{\theta^*(\psi^{n})}(t_j,X_{t_j,i}^{\alpha^{\psi^n}})^{\T}\left(\pa_{\psi_m}u^{\psi^n}(t_j,X_{t_j,i}^{\alpha^{\psi^n}})+\nabla_xu^{\psi^n}(t_j,X_{t_j,i}^{\alpha^{\psi^n}})\Delta_{t_j,i}^{\psi^n,\delta_n,m}\right)(t_{j+1}-t_j)\right).
			\end{align*}}
			\STATE Update the policy parameter: $\psi^{n+1}=\psi^n-\kappa_nP^n$ with $P^n=(P^n_1,\dots,P^n_p)^{\T}$.
			\ENDFOR
			\STATE \textbf{Output:} near-true model parameter $\theta^*$, near‑optimal policy parameter  $\psi^*$.
		\end{algorithmic}
		\caption{The SMP-based algorithm for model-free RL}
		\label{main_algorithm}
\end{algorithm}

Compared with the SMP model-based approach in Archibald et al.~\cite{YongSMP}, our framework is fully model-free. In \cite{YongSMP}, the unknown coefficients of the controlled system and the cost functional are assumed to admit a parametric representation, and the corresponding model parameters are first estimated from trajectory data via maximum likelihood. The estimated model is then used to construct the Hamiltonian and simulate its associated adjoint BSDE for policy gradient. In contrast, our method does not require any parametric representation of the system dynamics or the cost functions. Instead, for a given parameterized policy $\alpha^\psi$, the Hamiltonian gradient $H_a^{\alpha^\psi}(t,x)$ is learned directly from trajectory data through the decoupling-field representation of the adjoint BSDE. Thus, neither the model dynamics nor the cost functions need to be explicitly identified in our algorithm, while the SMP structure is directly exploited to obtain the policy gradient.
		
Compared with the continuous-time $q$-learning approach of Jia and Zhou~\cite{Zhou}, our method is based on the SMP rather than the $q$-function and does not require observations of the running cost along the trajectory. Jia and Zhou~\cite{Zhou} introduce the continuous-time $q$-function as a first-order approximation of the conventional $Q$-function and establish its connection with the instantaneous advantage rate and the Hamiltonian. Their $q$-learning framework jointly learns the $q$-function and the value function from observations of the running cost and the terminal cost. In contrast, our method directly learns the Hamiltonian gradient $H_a^{\alpha^\psi}$ associated with the SMP adjoint equation, which plays a similar role of the $q$-function in Jia and Zhou~\cite{Zhou}  and subsequently uses it to construct an inexact policy gradient. Moreover, the learning procedure only requires the total trajectory cost $J^{\alpha^\psi}=\int_0^T f(t,X_t^{\alpha^\psi},\alpha_t^\psi)dt
+g(X_T^{\alpha^\psi})$ rather than the running cost at each time point and the terminal cost separately. Hence, our framework can be applied when only the total performance of a trajectory is observable. However, this advantage comes at the cost of requiring the completion of an entire trajectory before a policy update can be performed. 

To facilitate comparison with the DPG approach of Cheng et al.~\cite{GuoDPG}, we first establish the connection between the SMP-based policy gradient \eqref{J_gradient} in Lemma \ref{lem:J_gradient} and their continuous-time DPG, which is formulated in terms of the advantage-rate function. Although the two gradient representations have rather different forms, they are in fact equivalent  under appropriate regularity conditions. For simplicity, we present the argument without discounting, which is consistent with the formulation adopted in this paper.  We also use the minimization convention of the present paper. If the reward-maximization convention of \cite{GuoDPG} is adopted, the corresponding running and terminal rewards can simply be replaced by their negatives. For a fixed policy $\alpha^{\psi}=(\alpha_t^{\psi})_{t\in[0,T]}$, let the corresponding value function be given by, for $(t,x)\in[0,T]\times\R^d$,
\begin{align*}
V^\psi(t,x)=\E\left[\int_t^Tf(s,X_s^{\alpha^\psi},u^\psi(s,X_s^{\alpha^\psi}))\d s+g(X_T^{\alpha^\psi})\big| X_t=x\right],
\end{align*}
and hence ${\cal J}(\psi)=\E[V^\psi(0,\xi)]$. We additionally impose the following differentiability condition on the value function.
	
\begin{ass}\label{ass3}
For any $\psi\in\R^p$, $V^{\psi}(t,x)$ is continuously differentiable in $t\in [0,T]$ and twice continuously differentiable in $x\in \R^d$. Furthermore, the gradient $\nabla_x V^{\psi}(t,x)$ is continuously differentiable in $t\in [0,T]$ and twice continuously differentiable in $x\in \R^d$.
\end{ass}
	
Following \cite{GuoDPG}, define the advantage-rate function associated with $u^\psi$ by
\begin{align}\label{eq:advantage-rate}
A^\psi(t,x,a)&:=\partial_t V^\psi(t,x)+b(t,x,a)^\top\nabla_xV^\psi(t,x) +\frac12\operatorname{tr}\left(\sigma(t,x,a)\sigma(t,x,a)^\top\nabla_{xx}^2V^\psi(t,x)\right)\nonumber\\
&\quad+f(t,x,a).
\end{align}
Under Assumptions \ref{ass1} and \ref{ass3},  the value function $V^{\psi}$ satisfies the following policy-evaluation equation given by, for $(t,x)\in[0,T]\times\R^d$,
\begin{align}\label{eq:policy-evaluation}
A^\psi(t,x,u^{\psi}(t,x))=0
\end{align}
with terminal condition $V^\psi(T,x)=g(x)$. The advantage-rate DPG of \cite{GuoDPG} takes the form given by, for $\psi\in\R^q$,
\begin{align}\label{eq:guo-dpg}
\nabla_\psi {\cal J}(\psi)=\E\left[ \int_0^T \nabla_\psi u^\psi(t,X_t^{\alpha^\psi})^\top\nabla_a A^\psi
(t,X_t^{\alpha^\psi},u^\psi(t,X_t^{\alpha^\psi}))\d t \right].
\end{align}
	
The following result shows that \eqref{J_gradient} and \eqref{eq:guo-dpg}  are two equivalent representations of the same policy gradient.
\begin{lemma}\label{lem:equivalence-dpg}
Under Assumptions~\ref{ass1} and \ref{ass3}, let  $(Y^{\alpha^\psi},Z^{\alpha^\psi})=(Y_t^{\alpha^\psi},Z_t^{\alpha^\psi})_{t\in[0,T]}$ be the unique solution to the BSDE \eqref{eq:BSDE}. Then, it holds that
\begin{align}\label{eq:dpg-equivalence-main}
&\E\left[\int_0^T\left(\nabla_{\psi}u^{\psi}(t,X_t^{\alpha^{\psi}})+\nabla_xu^{\psi}(t,X_t^{\alpha^{\psi}})\nabla_{\psi}X_t^{\alpha^{\psi}}\right)^{\T}\nabla_aH(t,X_t^{\alpha^{\psi}},\alpha_t^{\psi},Y_t^{\alpha^{\psi}},Z_t^{\alpha^{\psi}})\d t\right]\nonumber\\
&\qquad=\E\left[ \int_0^T \nabla_\psi u^\psi(t,X_t^{\alpha^\psi})^\top\nabla_a A^\psi
\left(t,X_t^{\alpha^\psi},u^\psi(t,X_t^{\alpha^\psi})\right)\d t \right].
\end{align}
Consequently, the SMP-based DPG formula \eqref{J_gradient} and the advantage-rate DPG formula \eqref{eq:guo-dpg} are equivalent.
\end{lemma}
	
Lemma \ref{lem:equivalence-dpg} shows that the difference between the two DPG formulas is a difference in representation rather than in the underlying policy gradient. The SMP representation given by \eqref{J_gradient} uses the total parameter sensitivity of the implemented control in the sense that
\begin{align*}
\nabla_{\psi} \left(u^\psi(t,X_t^{\alpha^\psi})\right)=\nabla_\psi u^\psi(t,X_t^{\alpha^\psi}) +\nabla_xu^\psi(t,X_t^{\alpha^\psi}) \nabla_\psi X_t^{\alpha^\psi},
\end{align*}
together with the open-loop SMP adjoint process $(Y,Z)=(Y_t,Z_t)_{t\in[0,T]}$ satisfying the BSDE \eqref{eq:BSDE}. In contrast, the advantage-rate representation absorbs the state-sensitivity term into the value-function-based closed-loop adjoint process given by
\begin{align*}
(\bar Y,\bar Z)=(\bar Y_t,\bar Z_t)_{t\in[0,T]}=(\nabla_xV^\psi(t,X_t^{\alpha^\psi}),\nabla_{x}^2V^\psi(t,X_t^{\alpha^\psi})\sigma(t,X_t^{\alpha^\psi},u^\psi(t,X_t^{\alpha^\psi})))_{t\in[0,T]},          
\end{align*}
and therefore involves only the partial policy derivative $\nabla_\psi u^\psi$. Importantly, this equivalence given by \eqref{eq:dpg-equivalence-main} holds after integration and expectation; the two integrands are not, in general, equal pathwise.
	
Although the two DPG formulas are equivalent, their model-free implementations are fundamentally different. The approach of Cheng et al.  \cite{GuoDPG} first learns the advantage-rate function $A^\psi$ through a martingale characterization, and then obtains $\nabla_aA^\psi$, by differentiating the learned critic $A^\psi$ with respect to the action. In contrast, our SMP-based approach directly approximates the Hamiltonian gradient $\nabla_aH(t,X_t^{\alpha^\psi},\alpha_t^\psi,Y_t^{\alpha^\psi},Z_t^{\alpha^\psi})$, which is precisely the quantity entering the SMP policy-gradient formula \eqref{J_gradient}. The remaining state sensitivity $\nabla_\psi X^{\alpha^\psi}$ is approximated directly from perturbed state trajectories.  This direct learning approach allows us to establish convergence results and error estimates for the resulting inexact-gradient policy optimization scheme, whereas Cheng et al.~\cite{GuoDPG} do not provide corresponding convergence guarantees. From an algorithmic perspective, this comes at the cost of additional exploration, since learning the Hamiltonian gradient requires multiple explorations, while Cheng et al.~\cite{GuoDPG} requires only a single exploration.  In particular, the SMP-based formulation avoids differentiating an approximately learned advantage function. This distinction is potentially important because convergence of an approximate critic to $A^\psi$ in function value does not, without additional regularity, automatically imply convergence of its action derivative to $\nabla_aA^\psi$.

\section{Main Results}\label{sec:main_result}
	
In this section, we present the theoretical results of the paper. First of all, we show the convergence of the inner loop.
\begin{prop}[Convergence of Inner Loop]\label{prop:inner_loop}
Let Assumption~\ref{ass1} and Assumption~\ref{ass2} hold. For any $n\geq1$ and any initial model parameter $\theta^0\in\R^q$, let $(\theta^k)_{k=0}^{\infty}$ be the sequence generated by the inner loop of Algorithm~\ref{main_algorithm} or equivalently by \eqref{theta_update}. Assume that the learning rate sequence $(\eta_k)_{k=1}^{\infty}$ and the hyperparameter sequence $(\epsilon_k)_{k=1}^{\infty}$ satisfy the following conditions given by
\begin{align}\label{inner_condition}
\lim_{k\to\infty}\eta_k=0,\quad\sum_{k=0}^{\infty}\eta_k=\infty,\quad
\sum_{k=0}^{\infty}\eta_k|\epsilon_k|^2<\infty.
\end{align}
Let $n\geq1$ be fixed. Then, the limit $L^{\psi^n}:=\lim_{k\to\infty}L(\theta^k;\psi^n)$ exists, and it holds that
\begin{align}\label{inner_convergence}	\lim_{k\to\infty}\left|\nabla_{\theta}L(\theta^k;\psi^n)\right|=0.
\end{align}
Furthermore, we have
\begin{align}\label{L_error}
&\left|L(\theta^k;\psi^n)-L^{\psi^n}\right|\nonumber\\
&\qquad\leq \frac74\sum_{\lambda=k}^{\infty}\eta_{\lambda}\left(\underbrace{\left|\nabla_{\theta}L(\theta^{\lambda};\psi)\right|^2}_{J_1}+2\underbrace{\left|\boldsymbol{e}_{\theta^{\lambda}}^{\psi^n}\right|^2}_{J_2}+\underbrace{2\left|\nabla_{\theta}L(\theta^{\lambda};\psi^n)+\boldsymbol{e}_{{\theta}^{\lambda}}^{\psi^n}-G^{\lambda}\right|^2}_{J_3}\right),
\end{align}
where, $\eta_{\lambda}$ is the learning rate in iteration $\lambda$; while $G^{\lambda}$ is introduced on line 10 in Algorithm~\ref{main_algorithm}.
\end{prop}
	
	\begin{remark}
		The three terms $J_1,J_2$ and $J_3$ in \eqref{L_error} correspond to different sources of error. In particular, the term $J_1$ arises from the error induced by the iterative algorithm. The term $J_2$ represents the error in the approximation of $\E[\int_0^T\pa_{\theta_l}h^{\theta^{\lambda}}
		(t,X_t^{\alpha^{\psi^n}})^{\T}H_a^{\alpha^{\psi^n}}(t,X_t^{\alpha^{\psi^n}})\d t]$ by using $\frac{J(\alpha^{\psi^n}+\epsilon_{\lambda}\pa_{\theta_l}h^{\theta^{\lambda}}(\cdot,X_{\cdot}^{\alpha^{\psi^n}}))-J(\alpha^{\psi^n})}{\epsilon_\lambda}$ for $l=1,\dots,q$. However, the term $J_3$ stands for the Monte Carlo discretization error.
	\end{remark}
	
	We next present our main result of the paper, which establishes the convergence of the overall algorithm.
	
	\begin{theorem}[Convergence of Algorithm~\ref{main_algorithm}]\label{thm:algorithm_convergence} Let Assumption~\ref{ass1} and \ref{ass2} hold. 
		For any initial policy parameter $\psi^0\in\R^p$, let $(\psi^n)_{n=1}^{\infty}$ be the sequence generated by Algorithm~\ref{main_algorithm}, or equivalently generated by \eqref{psi_update}. Let the learning rate sequence $(\kappa_n)_{n=1}^{\infty}$, the hyperparameter sequence $(\delta_n)_{n=1}^{\infty}$ and the residual loss $L(\theta^*(\psi^n);\psi^n)$ satisfy the following conditions:
\begin{align}\label{outer_condition}
\lim_{n\to\infty}\kappa_n=0,\quad\sum_{n=0}^{\infty}\kappa_n=\infty,\quad
			\sum_{n=0}^{\infty}\kappa_n\left(L(\theta^*(\psi^n);\psi^n)+|\delta_n|^2\right)<+\infty.
		\end{align}
		Then, the sequence $(\mathcal{J}(\psi^n))_{n=1}^{\infty}$ converges to a limit $\mathcal{J}^*$, i.e., $\mathcal{J}^*:=\lim_{n\to\infty}\mathcal{J}(\psi^n)$, and
		\begin{align}\label{outer_convergnece}
			\lim_{n\to\infty}\left|\nabla_{\psi}\mathcal{J}(\psi^n)\right|=0.
		\end{align}
		Furthermore, if one defines the Monte Carlo discretization error $\boldsymbol{\nu}_{N,N_T}^n$ for $n\geq1$ by
		\begin{align*}
			\boldsymbol{\nu}_{N,N_T}^{n}:=P^n-\E\left[\int_0^T\left(\nabla_{\psi}u^{\psi^n}(t,X_t^{\alpha^{\psi^n}})+\nabla_xu^{\psi^n}(t,X_t^{\alpha^{\psi^n}})\Delta_t^{\psi^n,\delta}\right)^{\T}h^{\theta^*(\psi^n)}(t,X_t^{\alpha^{\psi^n}},\alpha^{\psi^n}_t)\d t\right],
		\end{align*}
		we have the following error-decomposition:
		\begin{align}\label{J_error}
			\left|\mathcal{J}(\psi^n)-\inf_{\psi\in\R^p}\mathcal{J}(\psi)\right|&\leq \underbrace{\left|\mathcal{J}^*-\inf_{\psi\in\R^p}\mathcal{J}(\psi)\right|}_{K_1}+ \underbrace{\frac74\sum_{\lambda=n}^{\infty}\kappa_{\lambda}|\nabla_{\psi}\mathcal{J}(\psi^{\lambda})|^2}_{K_2}+\underbrace{\frac{21}{4}\sum_{\lambda=n}^{\infty}\kappa_{\lambda}|\boldsymbol{\nu}_{N,N_T}^{\lambda}|^2}_{K_3}\nonumber\\
			&\quad+C\left(\underbrace{\sum_{\lambda=n}^{\infty}\kappa_{\lambda}\delta_{\lambda}^2}_{K_4}+\underbrace{\sum_{\lambda=n}^{\infty}\kappa_{\lambda}L(\theta^*(\psi^{\lambda});\psi^{\lambda})}_{K_5}\right),
		\end{align}
		where, $C>0$ is a constant which depends on $T$ and $M$ introduced in Algorithm~\ref{main_algorithm}.
	\end{theorem}
	
	\begin{remark}
		The error decomposition in \eqref{J_error} consists of five distinct sources of error. Specifically, $K_1$ represents the modeling error associated with
		the parameterization of the policy. It measures the gap caused by the possibility that the optimal policy does not belong to the prescribed parametric policy family $\Psi$. The term $K_2$ corresponds to the convergence
		error of the gradient descent algorithm. The term $K_3$ arises from the Monte Carlo approximation used in the numerical implementation. The term $K_4$ is induced by approximating $\nabla_{\psi}X^{\alpha^{\psi^\lambda}}$ with the finite-difference process $\Delta^{\psi^\lambda,\delta_\lambda}$. In light of
		Lemma~\ref{lem:X_gradient_approx}, this approximation error is of order $\delta_\lambda^2$, and hence $K_4$ quantifies the cumulative error arising from this approximation. Lastly, $K_5$ stands for the learning error in approximating 	$H_a^{\alpha^{\psi^\lambda}}(t,X_t^{\alpha^{\psi^{\lambda}}})$, which measures the error introduced by
		learning the Hamiltonian gradient.
	\end{remark}

\section{Numerical Analysis}\label{sec:numerical}
This section presents numerical experiments demonstrating the practical efficacy of the proposed RL algorithm based on SMP. We first consider a linear-quadratic (LQ) stochastic control problem, and then examine the optimal portfolio allocation problem under exponential utility.

\subsection{LQ Stochastic Control Problem}\label{sec:LQ}
Consider the case  $d=r=l=1$. For any $\alpha\in\A[0,T]$,  the state process $X^{\alpha}=(X_t^{\alpha})_{t\in [0,T]}$ evolves as the following controlled dynamics, for $t\in(0,T]$, 
\begin{align}\label{SDE_dynamic-LQ}
\d X_t^{\alpha}=(\mu X_t^{\alpha}+\alpha_t)\d t+\sigma \d W_t,\quad X_0=x\in\R.
\end{align}
The corresponding cost functional is specified as:
\begin{align}\label{cost_func-LQ}
J(x;\alpha)=\E\left[\int_0^T\frac{1}{2}\left(A\left(X_t^{\alpha}\right)^2+ B\alpha_t^2\right) \d t+\frac{C}{2}\left(X_T^{\alpha}\right)^2\Big|X_0=x\right],
\end{align}
where $\mu,\sigma\in\R$ and $A,B,C>0$ are unknown constants. 

The agent's goal is to minimize the cost functional \eqref{cost_func-LQ} over all $\alpha\in\A[0,T]$. By using the SMP and standard LQ arguments, the optimal control and the solution to the adjoint BSDE \eqref{eq:BSDE} under the optimal control are respectively given by, for any $t\in[0,T]$,
\begin{align}
\alpha^*_t&=-\left(\mu+h+\frac{2h\ell e^{-2h(T-t)}}{1-\ell e^{-2h(T-t)}}\right)X_t^*,\label{eq:optimal-control-LQ}\\
Y_t&=B\left(\mu+h+\frac{2h\ell e^{-2h(T-t)}}{1-\ell e^{-2h(T-t)}}\right)X_t^*\label{eq:optimal-Y-LQ}
\end{align}
with the parameters $h:=\sqrt{\mu^2+\frac{A}{B}}$ and $\ell:=\frac{C-B(\mu+h)}{C-B(\mu-h)}$. Here, the process $X^*=(X_t^*)_{t\in[0,T]}$ is the state process under the optimal control $\alpha^*=(\alpha_t^*)_{t\in[0,T]}$, i.e., $X_t^*:=X_t^{\alpha^*}$ for $t\in[0,T]$. Motivated by the explicit structures in \eqref{eq:optimal-control-LQ}, \eqref{eq:optimal-Y-LQ} and the Hamiltonian
gradient $\nabla_a H(t,x,a,y,z)=y+Ba$, we adopt the following finite-dimensional parameterizations, for $(t,x)\in[0,T]\times\R$,
\begin{align}\label{eq:parameter-LQ}
\begin{cases}
\displaystyle u^{\psi}(t,x)=-\left(\psi_1+\frac{2\psi_2\psi_3e^{-2\psi_2(T-t)}}{1-\psi_3e^{-2\psi_2(T-t)}}\right)x,\\[1em]
\displaystyle h^{\theta}(t,x)=\left(\theta_1\theta_4+\frac{2\theta_2\theta_3\theta_4e^{-2\theta_2(T-t)}}{1-\theta_3e^{-2\theta_2(T-t)}}\right)x+\theta_4u^{\psi}(t,x),
\end{cases}
\end{align}
where, the corresponding target parameters are $\psi^*=( \psi_1^*,\psi_2^*,\psi_3^*)=(\mu+h,h,\ell)$ and $\theta^*=(\theta_1^*,\theta_2^*, \theta_3^*, \theta_4^*)=(\mu+h,h,\ell,B)$. With these choices, $u^{\psi^*}(t,x)$ coincides with the optimal feedback control, while $h^{\theta^*}(t,X_t^{\alpha^{\psi}})$ coincides with the true Hamiltonian gradient given by
$\nabla_a H(t,X_t^{\alpha^{\psi}},u^{\psi}(t,X_t^{\alpha^{\psi}}),Y_t^{\alpha^{\psi}},Z_t^{\alpha^{\psi}})$ along the corresponding LQ structure. We can verify that the Assumptions \ref{ass1} and \ref{ass2} are satisfied by choosing all parameter values from a compact set.

The parameters for simulation are set as follows. The time horizon is $T=1.0$, the number of time steps is $N_T=100$, and the discretization step is $\Delta t=T/N_T=0.01$. The underlying LQ dynamics and cost use $\mu=0.5$, $\sigma=0.2$, $A=1.0$, $B=0.2$ and $C=1.0$; while the initial state is sampled as $X_0\sim\mathcal {\cal N}(2.0,0.5^2)$. We use a batch size of $N=512$ trajectories per outer iteration and an evaluation batch size of $N_{eval}=2048$. The algorithm runs for $N_{iter}=1200$ outer iterations, with $N_{inner}=4$ inner critic updates per outer iteration. The initial  parameters are $\psi_0=(3.0,2.3,0.35)$ and $\theta_0=(3.0,2.3,0.35,0.1)$. The learning-rate and perturbation schedules are provided by, for $k,n\geq1$,
\begin{align*}
\eta_k=\frac{2.0}{(k+20)^{0.65}},\quad
\kappa_n=\frac{12.0}{(n+20)^{0.65}},\quad
\epsilon_k=\frac{0.02}{(k+20)^{0.30}},\quad
\delta_n=\frac{0.05}{(n+20)^{0.30}}.
\end{align*}

We simulate the state process and monitor the convergence of the parameter iterates generated by Algorithm~\ref{main_algorithm}. Figure~\ref{fig:LQ-convergence} illustrates the evolution of selected parameters during training, demonstrating that they indeed converge toward their true values. This confirms that the proposed model-free learning scheme is capable of recovering the optimal solution without knowledge of the system dynamics. The convergence is observed to be stable, with the parameters settling to their targets after a transient phase.

Figure~\ref{fig:LQ-policy}-(a) compares the learned policy with the true optimal policy over the entire time horizon. The learned gain closely tracks the optimal one, indicating that the policy parameterization is sufficiently expressive and that the SMP-DPG updates correctly identify the optimal control law. The cost history on matched evaluation paths is reported in Figure~\ref{fig:LQ-policy}-(b), which confirms that the learned policy's performance improves throughout training and converges to the optimal cost. The gap between the learned and optimal costs narrows steadily and eventually becomes negligible, further validating the effectiveness of the algorithm.
\begin{figure}[htbp]
\centering
  \subfigure[]{
        \includegraphics[width=7cm]{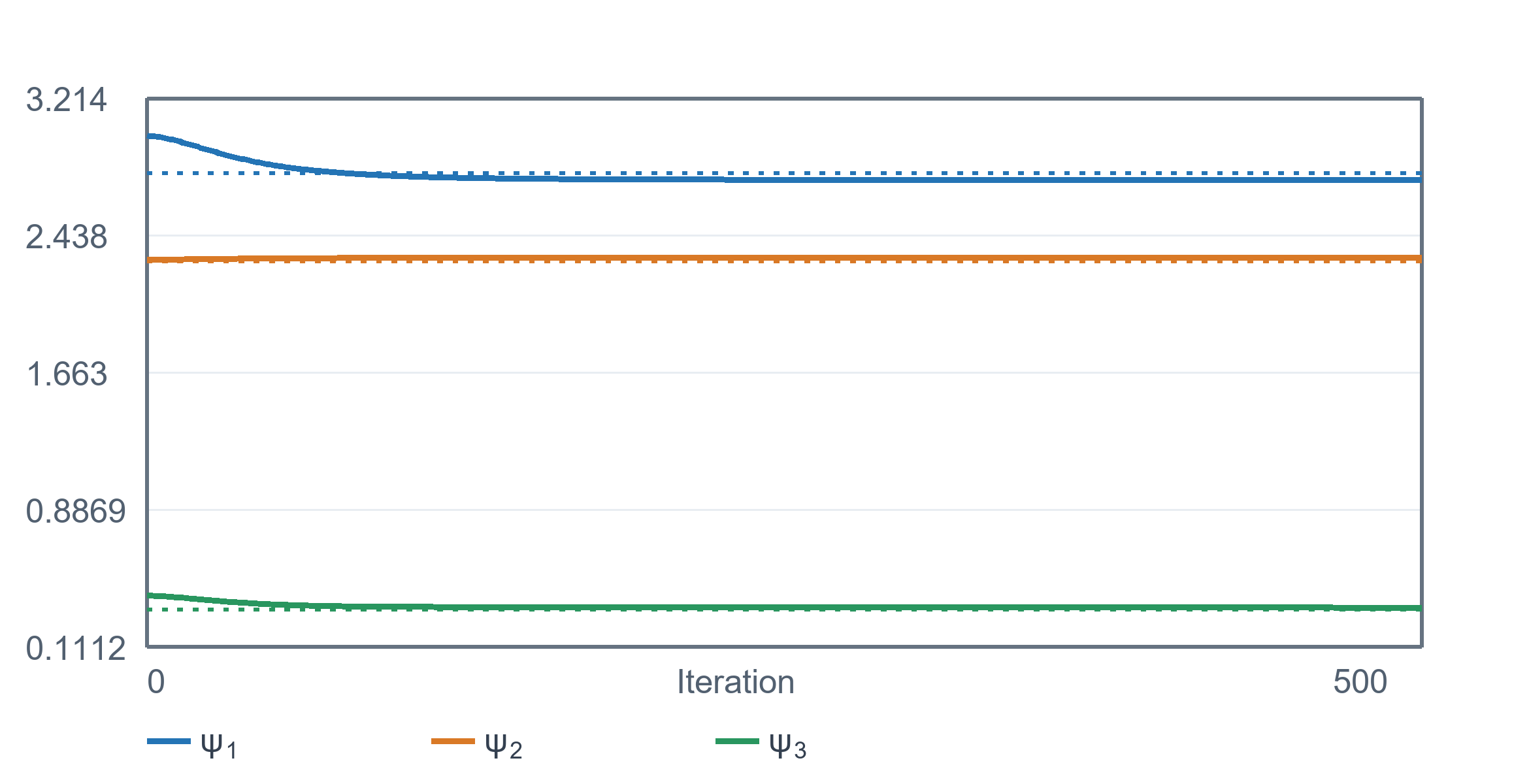}
    }
  \subfigure[]{
        \includegraphics[width=7cm]{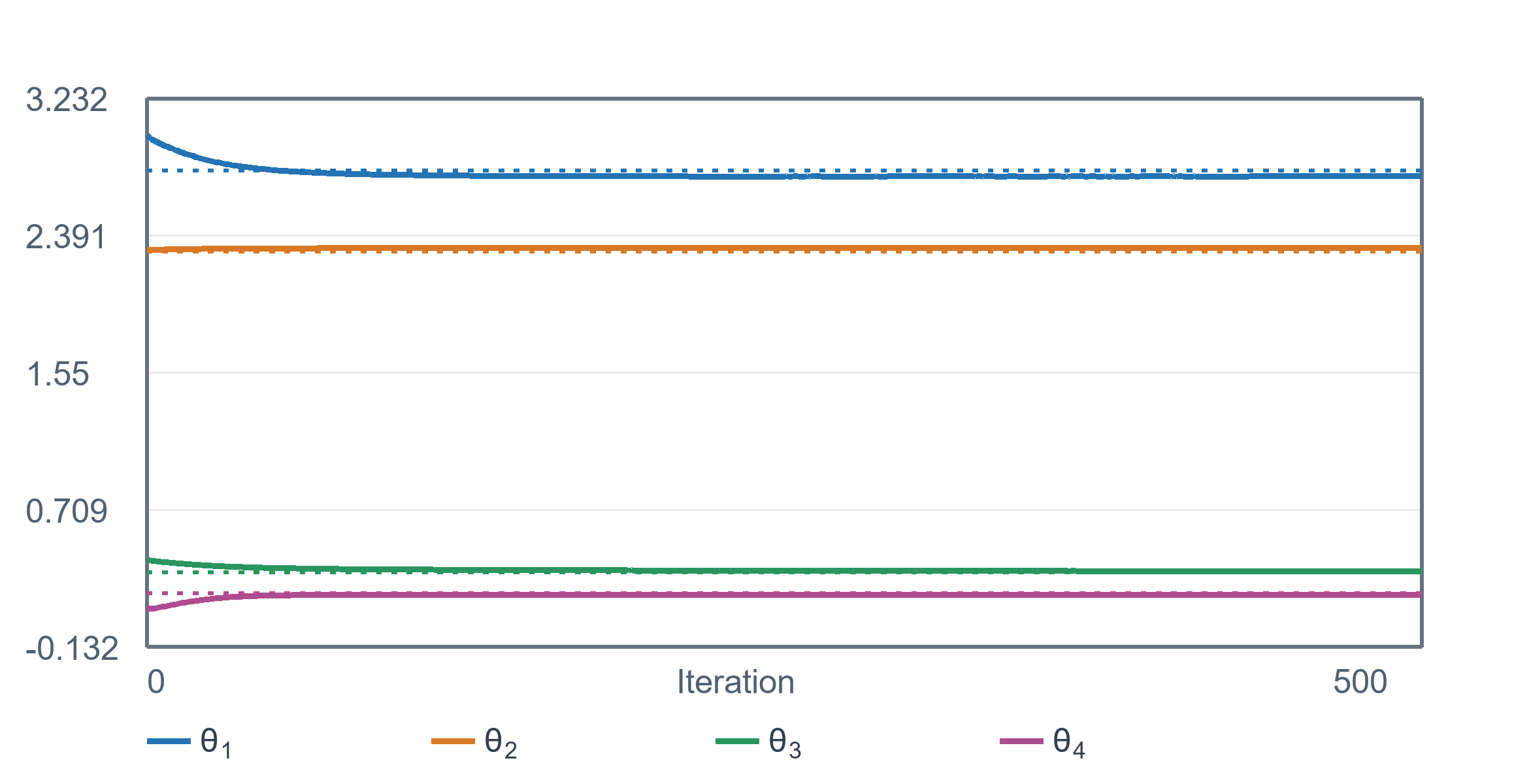}
    }
\caption{ Convergence of parameter iterations using Algorithm \ref{main_algorithm}. The solid line represents the learned parameter values during the iterations, while the dashed line represents the true parameter values.}\label{fig:LQ-convergence}
\end{figure}

\begin{figure}[htbp]
\centering
  \subfigure[]{
        \includegraphics[width=7cm]{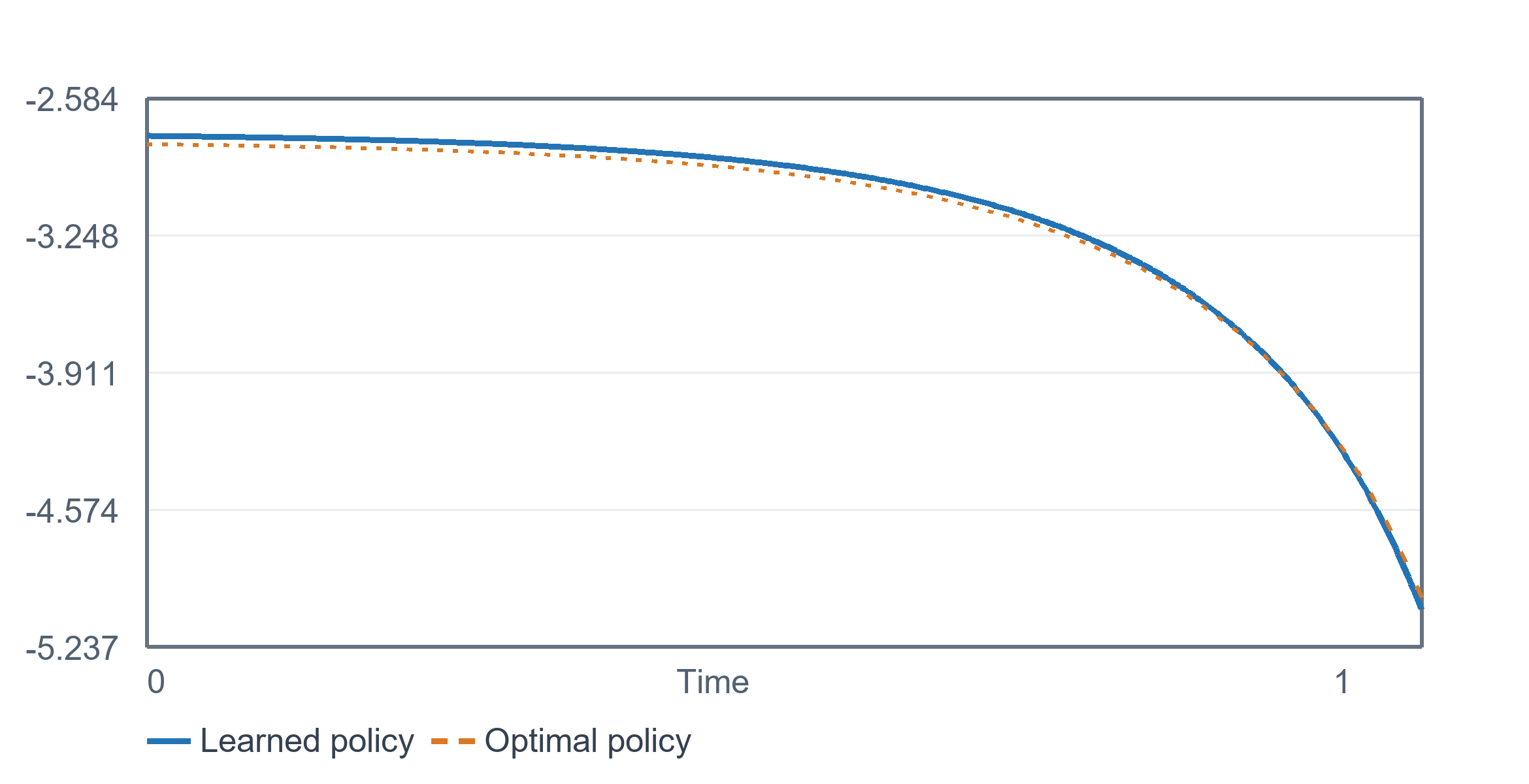}
    }
  \subfigure[]{
        \includegraphics[width=7cm]{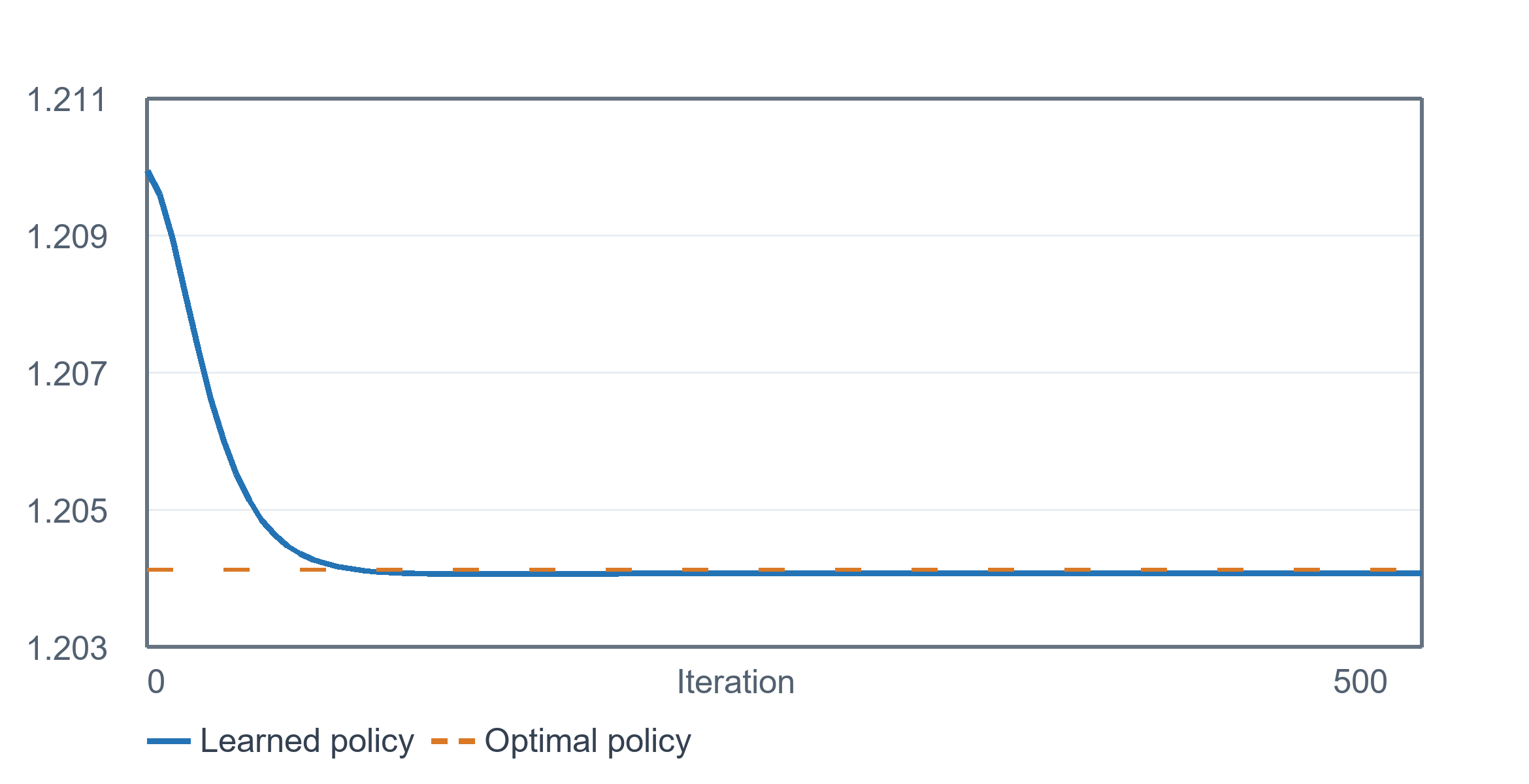}
    }
\caption{Left panel (a): The learnt policy vs the true optimal policy. The state level is set as $x=1$. Right panel (b): The cost history of the learnt policy vs the true optimal policy. }\label{fig:LQ-policy}
\end{figure}

\subsection{Optimal Portfolio Allocation Problem}

In this example, we consider a financial market consisting of a risky asset and a risk-free money account. The price process $S=(S_t)_{t\in[0,T]}$ of the risky asset and the price process $B=(B_t)_{t\in[0,T]}$ of the risk-free money account follow the following dynamics respectively
\begin{align*}
 \d S_t=\mu S_t \d t+\sigma S_t \d W_t, \quad
 \d B_t=r B_t \d t,
\end{align*}
where, the return rate $\mu> 0$, the volatility $\sigma > 0$ and the risk-free interest rate $r>0$ are assumed to be unknown. 

At each time $t\in[0,T] $, let $\alpha_t$ denote the amount of wealth invested in the risky asset $S$.  Under an admissible control $\alpha=(\alpha_t)_{t\in[0,T]}\in\A[0,T]$, the corresponding self-financing wealth process $X^{\alpha}=(X_t^{\alpha})_{t\in[0,T]}$ evolves according to the following dynamics, for $t\in(0,T]$,
\begin{align*}
\d X_t^{\alpha}=\left(r X_t^{\alpha}+(\mu-r) \alpha_t\right)\d t+\sigma \alpha_t \d W_t,\quad X_0=x>0.
\end{align*}
The agent is endowed with the exponential utility function $U(x)=-\gamma e^{-\gamma x}$ with the risk-aversion parameter $\gamma>0$ being also assumed to be unknown. Then, the agent seeks to maximize the expected utility of terminal wealth at time $T$. Equivalently, the agent minimizes the following cost functional over all admissible controls $\alpha\in\A[0,T]$,
\begin{align}\label{eq:exam2J}
J(x;\alpha)=\E\left[-U(X_T^{\alpha})| X_0=x\right].
\end{align}

The simulation parameters are set as follows. The time horizon is $T=1.0$ with $N_T=100$ time steps and $\Delta t=T/N_T=0.01$.  We set $r=0.02$, $\mu=0.20$,  $\sigma=0.10$ and the risk-aversion parameter $\gamma=1.0$. Initial wealth is uniformly sampled from the interval $[1.0,2.0]$.  The policy function $\mu^\psi(t,x)$ and the auxiliary Hamiltonian gradient function $\tilde{h}^\theta(t,x,a)$ are implemented by a {\it two-layer multilayer perceptron (MLP)} with ReLU activations and hidden dimension 32. We then parameterize the Hamiltonian gradient function as $h^\theta(t,x)=\tilde{h}^\theta(t,x,\mu^\psi(t,x))$.

We use a batch size of $N=128$ trajectories per outer iteration and an evaluation batch size of $N_{eval}=4096$. The algorithm runs for $N_{iter}=150$ outer iterations, with $N_{inner}=8$ inner critic updates per outer iteration.  For $n,k\geq1$,  the learning-rate and perturbation schedules are given by
\begin{align*}
\eta_k=\frac{0.25}{(k+20)^{0.65}},\quad
\kappa_n=\frac{100.0}{(n+20)^{0.65}},\quad
\epsilon_k=\frac{0.02}{(k+20)^{0.30}},\quad
\delta_n=\frac{0.05}{(n+20)^{0.30}}.
\end{align*}

We simulate the controlled state process and evaluate the policies generated by Algorithm~\ref{main_algorithm}. Figure~\ref{fig:portfolio-policy}-(a) compares the learned policy with the true optimal policy over the entire time horizon, while Figure~\ref{fig:portfolio-policy}-(b) illustrates the cost history. As shown in Figure~\ref{fig:portfolio-policy}, the learned policy closely tracks the optimal one, and its performance improves steadily throughout training, with the cost converging to the optimal value.  The cost history shows a stable downward trend, with the objective value decreasing rapidly in the early iterations and then stabilizing near the optimum, which provides empirical evidence for the convergence of algorithm.

Since a neural network parameterization is used in this example, Assumptions~\ref{ass1} and~\ref{ass2} are generally difficult to verify or satisfy. However, these assumptions are introduced only to derive the theoretical convergence guarantee for the algorithm; they are not required for the implementation or execution of the algorithm itself. In other words, they serve as sufficient analytical conditions rather than operational requirements. The numerical results of this example therefore demonstrate the robustness and practical effectiveness of the proposed algorithm beyond the idealized theoretical setting. This also suggests that, although the assumptions may be violated under neural network approximation, the algorithm can still perform reliably in practice, highlighting its potential for broader applications.
\begin{figure}[htbp]
\centering
  \subfigure[]{
        \includegraphics[width=7cm]{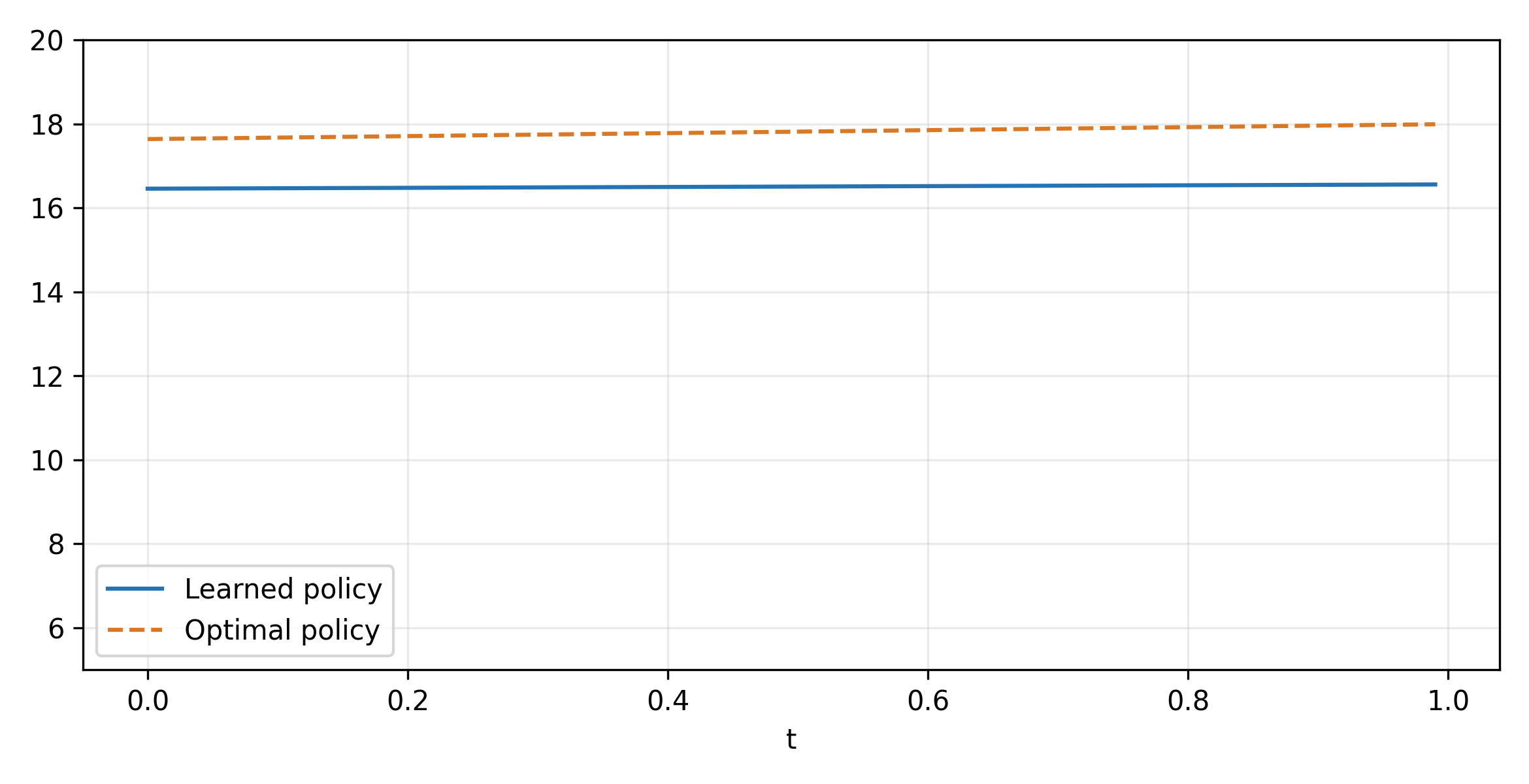}
    }
  \subfigure[]{
        \includegraphics[width=7cm]{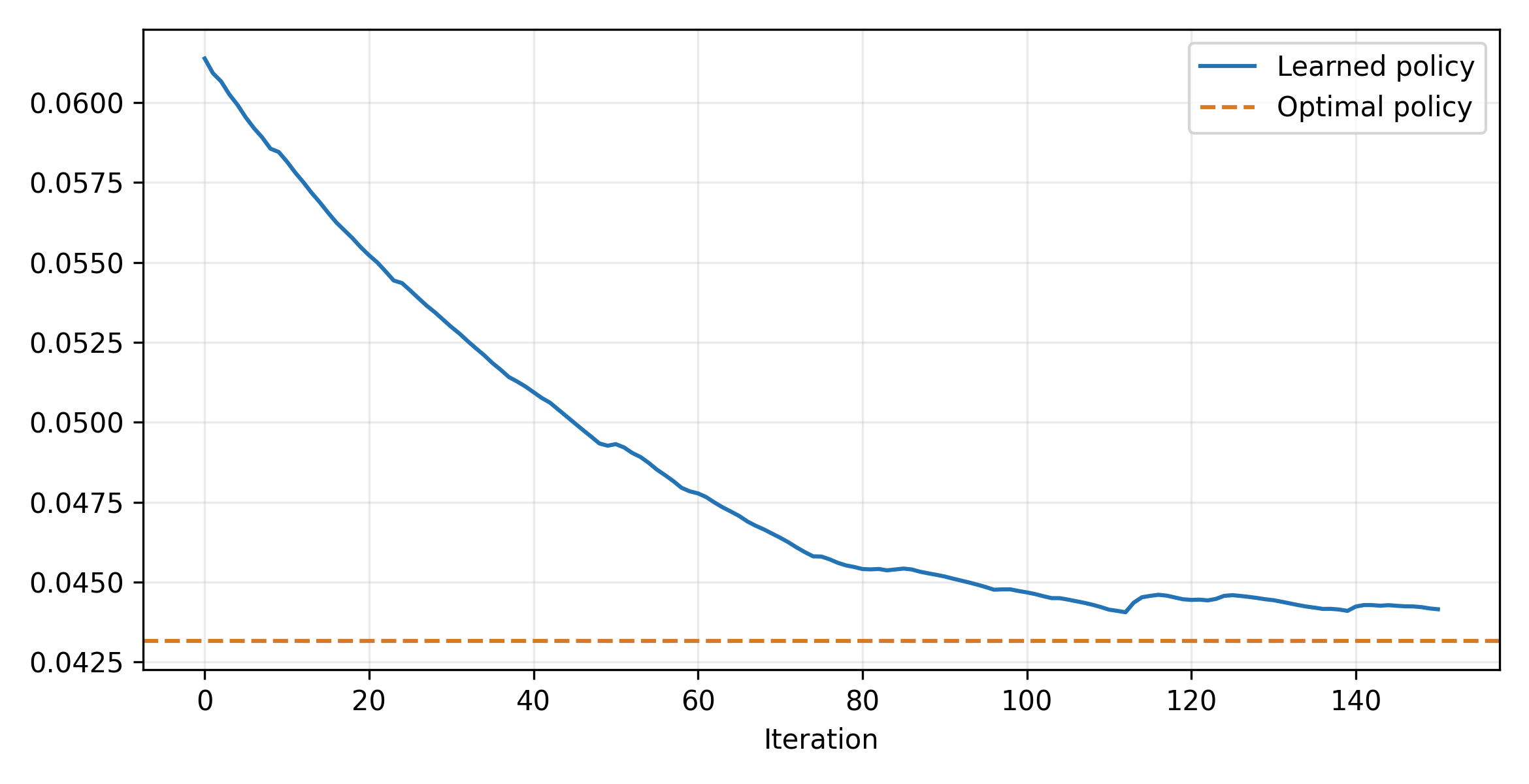}
    }
\caption{Left panel (a): The learnt policy vs the true optimal policy.  The state level is set as $x=1$. Right panel (b): The cost history of the learnt policy vs the true optimal policy. }\label{fig:portfolio-policy}
\end{figure}

\section{Conclusion}\label{sec:conclusion}
In this paper, we developed a DPG framework for continuous-time stochastic control based on the SMP. Our goal is not to propose an algorithm that is computationally more efficient than existing RL methods, but rather to provide an alternative SMP-based route to policy optimization. The proposed method learns the Hamiltonian gradient directly and uses it to construct the policy update. We also establish convergence results for both the inner-loop learning of the Hamiltonian gradient and the outer policy iteration, thereby providing theoretical justification for the proposed framework.

A natural direction for future research is to extend the proposed method to stochastic control problems with general dynamic (state-control) constraints (cf. Bo et al.~\cite{BWY}). Such problems are often difficult to address within the standard HJB framework, whereas the variational structure of the SMP may still provide a useful foundation for developing policy gradient methods. Although the corresponding SMP has been derived in \cite{BWY}, we are currently unable to establish a decoupling field for the associated adjoint BSDE as in Lemma~\ref{lem:decoupling_field}, and hence Algorithm~\ref{main_algorithm} cannot be directly applied. In particular, since the adjoint processes $(Y,Z)=(Y_t,Z_t)_{t\in[0,T]}$ are not directly observable from trajectory data, they need to be represented in terms of the time, state and control. Establishing such a representation, together with the corresponding theoretical guarantees, therefore requires further analysis and constitutes an important direction for future research on constrained RL.

	\appendix
	
	\section{Proofs of Main Results}\label{sec:proof_main} 
	
	This section provides the proofs for the main theoretical results in Section~\ref{sec:main_result}.
	\begin{proof}[Proof of Proposition~\ref{prop:inner_loop}]
		We denote by $\boldsymbol{\mu}_{N,N_T}^{k,n}$ the Monte Carlo discretization error, i.e., for $k,n\in\mathbb{N}$
		\begin{align}\label{MC_error}
			\boldsymbol{\mu}_{N,N_T}^{k,n}=G^{n,k}-\nabla_{\theta}L(\theta^k;\psi^n)-\boldsymbol{e}^{\psi^n}_{\theta^k}.
		\end{align}
		We have from the LLN that $\lim_{N,N_T\to\infty}\boldsymbol{\mu}_{N,N_T}^{k,n}=0$, $\Pb$-a.s.. Differentiating $L(\theta;\psi^n)$ twicely with respect to $\theta$, we obtain that
		\begin{align*}	\pa_{\theta_{l_1}\theta_{l_2}}^2L(\theta;\psi^n)&=\E\left[\int_0^T\pa_{\theta_{l_1}}h^{\theta}(t,X_t^{\alpha^{\psi^n}},\alpha_t^{\psi^n})^{\T}\pa_{\theta_{l_2}}h^{\theta}(t,X_t^{\alpha^{\psi^n}},\alpha_t^{\psi^n})\d t\right]\\
			&\quad+\E\left[\left(h^{\theta}(t,X_t^{\alpha^n},\alpha_t^n)-H^{\alpha^n}(t,X_t^{\alpha^n},\alpha^n_t)\right)^{\T}\pa_{\theta_{l_1}\theta_{l_2}}^2h^{\theta}(t,X_t^{\alpha^n},\alpha_t^n)\d t\right].
		\end{align*}
		By using Lemma~\ref{lem:X_moment}, Lemma~\ref{lem:decoupling_field} and Assumption~\ref{ass2}, there exists a constant $C>0$ which depends on $T$ and $M$ introduced in Assumption~\ref{ass1} such that
		\begin{align}\label{Hessian_bound_inner}
			\left|	\pa_{\theta_{l_1}\theta_{l_2}}^2L(\theta;\psi^n)\right|\leq 6\E\left[\int_0^TM_1(t)^2\left(1+|X_t^{\alpha^{\psi}}|^4   \right)\d t\right]\leq C.
		\end{align}
		In a similar fashion, we can also derive that
		\begin{align}\label{gradient_bound_inner}
			\left|\nabla_{\theta}L(\theta;\psi^n)\right|\leq C.
		\end{align}
		Hence, we have from Taylor's expansion and Cauchy-Schwartz inequality that
		\begin{align*}
			&L(\theta^{k+1};\psi^n)-L(\theta^k;\psi^n)\leq \nabla_{\theta}L(\theta^k;\psi^n)^{\T}\left(\theta^{k+1}-\theta^k\right)+\frac{C}{2}\left|\theta^{k+1}-\theta^k\right|^2\\
			=&-\eta_k\nabla_{\theta}L(\theta^k;\psi^n)^{\T}\left(\nabla_{\theta}L(\theta^k;\psi^n)+\boldsymbol{e}_{\theta^k}^{\psi^n}+\boldsymbol{\mu}_{N,N_T}^{k,n}\right)+\frac{C\eta_k^2}{2}\left|\nabla_{\theta}L(\theta^k;\psi^n)+\boldsymbol{e}_{\theta^k}^{\psi^n}+\boldsymbol{\mu}_{N,N_T}^{k,n}\right|^2\\
			\leq&-\left(\frac{\eta_k}{2}-C\eta_k^2\right)\left|\nabla_{\theta}L(\theta^k;\psi^n)\right|^2+\left(\frac{\eta_k}{2}+C\eta_k^2\right)\left|\boldsymbol{e}_{\theta^k}^{\psi^n}+\boldsymbol{\mu}_{N,N_T}^{k,n}\right|^2.
		\end{align*}
		Summing the above equation from $0$ to $k$, we obtain from the fact $L(\theta;\psi)\geq 0$ that
		\begin{align}\label{sum_bound_inner}
			\sum_{\lambda=0}^{k}\left(\frac{\eta_{\lambda}}{2}-C\eta_{\lambda}^2\right)\left|\nabla_{\theta}L(\theta^{\lambda};\psi^n)\right|^2\leq L(\theta^0;\psi^n)+\sum_{\lambda=0}^k\left(\frac{\eta_{\lambda}}{2}+C\eta_{\lambda}^2\right)\left|\boldsymbol{e}_{\theta^{\lambda}}^{\psi^n}+\boldsymbol{\mu}_{N,N_T}^{\lambda,n}\right|^2.
		\end{align}
		By using \eqref{inner_condition}, we have that, the sum on the right hand of \eqref{sum_bound_inner} is convergent as $k\to\infty$ (we can choose both $N$ and $N_T$ large enough such that $|\boldsymbol{\mu}_{N,N_T}^{\lambda,n}|\leq |\boldsymbol{e}_{\theta^{\lambda}}^{\psi^n}|$), and hence \begin{align}\label{series_convergent_inner}
			\sum_{k=0}^{\infty}\left(\frac{\eta_k}{2}-C\eta_k^2\right)\left|\nabla_{\theta}L(\theta^k;\psi^n)\right|^2<\infty.
		\end{align}
		Moreover, thanks to \eqref{Hessian_bound_inner} and \eqref{gradient_bound_inner}, there exists a constant $C>0$ which depends on $T$ and $M$ introduced in Assumption~\ref{ass1} such that
        \begin{align*}
			\left|\nabla_{\theta}L(\theta_{k+1};\psi^n)-\nabla_{\theta}L(\theta_k;\psi^n)\right|\leq C\eta_k.
		\end{align*}
		Since $\lim_{k\to\infty}\eta_k=0$, we may assume WLOG that $\eta_k<1/(2C)$ for all $k\geq1$ by discarding finitely many initial terms of $(\eta_k)_{k=0}^{\infty}$ if necessary. Therefore, we can conclude that
		\begin{align*}
			\lim_{k\to\infty}|\nabla_{\theta}L(\theta^k;\psi^n)|=0.
		\end{align*}
		Here, we have applied the following lemma whose proof is delegated to Appendix~\ref{sec:proof_auxi}.
		\begin{lemma}\label{lem:auxi_lemma}
			Let $(a_n)_{n=1}^{\infty}$ be a sequence of positive reals such that $\sum_{n=1}^{\infty}r_na_n^2<\infty$ and $|a_{n+1}-a_n|\leq Cr_n$ for any $n\geq1$, where $C>0$ is a constant and $r_n>0$ satisfies $\sum_{n=1}^{\infty}r_n=\infty$. Then, we have $\lim_{n\to\infty}a_n=0$.
		\end{lemma}
On the other hand, by applying Taylor's expansion again, we have, for $k\geq1$,
\begin{align*}
L(\theta^{k+1};\psi^n)-L(\theta^k;\psi^n)&\geq \nabla_{\theta}L(\theta^k;\psi^n)^{\T}\left(\theta^{k+1}-\theta^k\right)-\frac{C}{2}\left|\theta^{k+1}-\theta^k\right|^2\\
&\geq -\left(\frac{3\eta_k}{2}+C\eta_k^2\right)\left|\nabla_{\theta}L(\theta^k;\psi^n)\right|^2-\left(\frac{\eta_k}{2}+C\eta_k^2\right)\left|\boldsymbol{e}_{\theta^k}^{\psi^n}+\boldsymbol{\mu}_{N,N_T}^{k,n}\right|^2.
\end{align*}
By using the assumption $\eta_k<1/(2C)$, the above inequality implies that, for $k,c\geq1$,
\begin{align*}
\left|L(\theta^{c};\psi^n)-L(\theta^k;\psi^n)\right|&\leq \frac74\sum_{\lambda=k}^{c-1}\eta_{\lambda}\left(\left|\nabla_{\theta}L(\theta^{\lambda};\psi^n)\right|^2+\left|\boldsymbol{e}_{\theta^{\lambda}}^{\psi^n}+\boldsymbol{\mu}_{N,N_T}^{\lambda,n}\right|^2\right)\\
&\leq \frac74\sum_{\lambda=k}^{c-1}\eta_{\lambda}\left(\left|\nabla_{\theta}L(\theta^{\lambda};\psi^n)\right|^2+2\left|\boldsymbol{e}_{\theta^{\lambda}}^{\psi^n}\right|^2+2\left|\boldsymbol{\mu}_{N,N_T}^{\lambda,n}\right|^2\right),
\end{align*}
where, we used Cauchy-Schwartz's inequality in the 2nd inequality. Thanks to the convergence of $\sum_k\eta_k|\nabla_{\theta}L(\theta^k;\psi^n)|^2$ and $\sum_k\eta_k|e_k|^2$, we can conclude that $\{L(\theta^k;\psi^n)\}_{k=1}^{\infty}$ forms a Cauchy sequence. We further define $L^{\psi^n}=\lim\limits_{k\to\infty}L(\theta^k;\psi^n)$ and we can deduce the following error decomposition:
\begin{align*}
\left|L(\theta^k;\psi^n)-L^{\psi^n}\right|\leq \frac74\sum_{\lambda=k}^{\infty}\eta_{\lambda}\left(\left|\nabla_{\theta}L(\theta^{\lambda};\psi^n)\right|^2+2\left|\boldsymbol{e}_{\theta^{\lambda}}^{\psi^n}\right|^2+2\left|\boldsymbol{\mu}_{N,N_T}^{\lambda,n}\right|^2\right).
\end{align*}
The proof is therefore complete.
	\end{proof}
	
	\begin{proof}[Proof of Theorem~\ref{thm:algorithm_convergence}]
		We first show that $\psi\to \nabla_{\psi}\mathcal{J}(\psi)$ is globally Lipschitz continuous. In fact, it follows from the chain rule that
		\begin{align}\label{chain_rule}
			\nabla_{\psi}\mathcal{J}(\psi)=\E\left[\int_0^T\left(\nabla_{\psi}\alpha^{\psi}_t\right)^{\T}\nabla_{\alpha}J(\alpha^{\psi})(t)\d t\right],
		\end{align}
where, $\nabla_{\alpha}J(\alpha^{\psi})(\cdot)=\nabla_a H(\cdot,X_{\cdot}^{\alpha^{\psi}},\alpha^{\psi}_{\cdot},Y_{\cdot}^{\alpha^{\psi}},Z_{\cdot}^{\alpha^{\psi}})$ denotes Fr\'{e}chet derivative of the objective functional $J(\alpha)$ defined by \eqref{cost_func} with respect to $\alpha$, evaluated at $\alpha^{\psi}$. The next lemma establishes the Lipschitz continuity of the Fr\'{e}chet derivative $\alpha\to\nabla_{\alpha} J(\alpha)$, whose proof is postponed to Appendix~\ref{sec:proof_auxi}.
		\begin{lemma}\label{lem:Hastable}
			The Fr\'{e}chet derivative $\nabla_{\alpha}J(\alpha)(\cdot)=\nabla_a H(\cdot,X_{\cdot}^{\alpha},\alpha_{\cdot},Y_{\cdot}^{\alpha},Z_{\cdot}^{\alpha})$ of the objective functional $J(\alpha)$ defined by \eqref{cost_func} is Lipschitz with respect to $\alpha$ in $\Psi$, i.e., there exists a constant $C>0$ only depending on $T$ and $M$ introduced in Assumption \ref{ass1} such that, for all $\psi^1,\psi^2\in\R^p$, 
			\begin{align}\label{H_Lipschitz}
				\left\|\nabla_{\alpha}J(\alpha^{\psi^1})-\nabla_{\alpha}J(\alpha^{\psi^2})\right\|_{L^2}\leq C\left\|\alpha^{\psi^1}-\alpha^{\psi^2}\right\|_{L^2}.
			\end{align}
		\end{lemma}	
		On the other hand, the next lemma shows that the gradient process $\nabla_{\psi}X^{\alpha^{\psi}}=(\nabla_{\psi}X_t^{\alpha^{\psi}})_{t\in [0,T]}$ defined in \eqref{X_gradient} is Lipschitz continuous with respect to $\psi\in\R^p$. Its proof is reported to Appendix~\ref{sec:proof_auxi}.
		\begin{lemma}\label{lem:gradientstable}
			There exists a constant $C>0$ only depending on $T$ and $M$ introduced in Assumption \ref{ass1} such that, for all $\psi^1,\psi^2\in\R^p$,
			\begin{align}\label{gradient_Lipschitz}
				\E\left[\sup_{t\in [0,T]}\left|\nabla_{\psi}X^{\alpha^{\psi^1}}_t-\nabla_{\psi}X_t^{\alpha^{\psi^2}}\right|^2\right]\leq C\left|\psi^1-\psi^2\right|^2.
			\end{align}
		\end{lemma} 
		As a direct consequence of Lemma~\ref{lem:X_gradient_approx} and Lemma~\ref{lem:gradientstable}, there exists a deterministic (positive) function $t\to C(t)$ satisfying $C:=\int_0^T C(t)\d t<\infty$, where the constant $C>0$ depends only on $T$ and $M$ introduced in Assumption \ref{ass1} such that, for any $t\in [0,T]$ and $\psi^1,\psi^2\in\R^p$,
		\begin{align*}
			&\left|\nabla_{\psi}\alpha^{\psi^1}_t-\nabla_{\psi}\alpha^{\psi^2}_t\right|\leq \left|\nabla_{\psi}u^{\psi^1}(t,X_t^{\alpha^{\psi^1}})-\nabla_{\psi}u^{\psi^2}(t,X_t^{\alpha^{\psi^2}})\right|\\
			&\qquad\qquad+\left|\nabla_xu^{\psi^1}(t,X_t^{\alpha^{\psi^1}})\nabla_{\psi}X_t^{\alpha^{\psi^1}}-\nabla_xu^{\psi^2}(t,X_t^{\alpha^{\psi^2}})\nabla_{\psi}X_t^{\alpha^{\psi^2}}\right|\leq C(t)\left|\psi^1-\psi^2\right|.
		\end{align*}
		Therefore, we derive that
		\begin{align}\label{alpha_Lipschitz}
			\left\|\nabla_{\psi}\alpha^{\psi^1}-\nabla_{\psi}\alpha^{\psi^2}\right\|_{L^2}\leq C\left|\psi^1-\psi^2\right|.
		\end{align}
		Consequently, by combing \eqref{chain_rule}, \eqref{H_Lipschitz} and \eqref{alpha_Lipschitz} altogether, we conclude that, for all $\psi^1,\psi^2\in\R^p$,
		\begin{align}\label{nablaJ_Lipschitz}
			&\left|\nabla_{\psi}\mathcal{J}(\psi^1)-\nabla_{\psi}\mathcal{J}(\psi^2)\right|\leq \E\left[\int_0^T\left|\nabla_{\psi}\alpha^{\psi^1}_t-\nabla_{\psi}\alpha^{\psi^2}_t\right|\left|\nabla_{\alpha}\mathcal{J}(\alpha^{\psi^1})(t)\right|\d t\right]\nonumber\\
			&\qquad\qquad+\E\left[\int_0^T\left|\nabla_{\psi}\alpha^{\psi^2}_t\right|\left|\nabla_{\alpha}\mathcal{J}(\alpha^{\psi^1})(t)-\nabla_{\alpha}\mathcal{J}(\alpha^{\psi^1})(t)\right|\d t\right]\nonumber\\
			&\qquad\quad\leq \|\nabla_{\alpha}\mathcal{J}(\alpha^{\psi^1})\|_{L^2}\|\nabla_{\psi}\alpha^{\psi^1}-\nabla_{\psi}\alpha^{\psi^2}\|_{L^2}+\|\nabla_{\psi}\alpha^{\psi^1}\|_{L^2}\|\nabla_{\alpha}\mathcal{J}(\alpha^{\psi^1})-\nabla_{\alpha}\mathcal{J}(\alpha^{\psi^2})\|_{L^2}\nonumber\\
			&\qquad\quad\leq C\left|\psi^1-\psi^2\right|, 
		\end{align}
		where the constant $C>0$ depends only on $T$ and $M$ introduced in Assumption \ref{ass1}. Here, we applied H\"{o}lder's inequality in the 2nd inequality and used the uniform $L^2$-boundedness, which follows from Assumption~\ref{ass1} and Lemma~\ref{lem:decoupling_field}. More precisely, there exists a constant $C>0$ depending only on $T$ and $M$ such that
		\begin{align*}
			\left\|\nabla_{\alpha}J(\alpha^{\psi})\right\|_{L^2} +
			\left\|\nabla_{\psi}\alpha^{\psi}\right\|_{L^2} \leq C, \quad \forall\psi\in\mathbb{R}^p.
		\end{align*}
		Consequently, we apply Taylor's expansion and Cauchy-Schwartz's inequality to derive that
		\begin{align}\label{J_bound}
			\mathcal{J}(\psi^{n+1})&-\mathcal{J}(\psi^n)\leq \nabla_{\psi}J(\psi^n)^{\T}\left(\psi^{n+1}-\psi^n\right)+\frac{C}{2}|\psi^{n+1}-\psi^n|^2\nonumber\\
			&\quad=-\kappa_n\nabla_{\psi}\mathcal{J}(\psi^n)^{\T}\left(\nabla_{\psi}\mathcal{J}(\psi^n)+\boldsymbol{E}^{\psi^n}+\boldsymbol{\nu}_{N,N_T}^n\right)+\frac{C\kappa_n^2}{2}\left|\nabla_{\psi}\mathcal{J}(\psi^n)+\boldsymbol{E}^{\psi^n}+\boldsymbol{\nu}_{N,N_T}^n\right|^2\nonumber\\
			&\quad\leq -\left(\frac{\kappa_n}{2}-C\kappa_n^2\right)\left|\nabla_{\psi}\mathcal{J}(\psi^n)\right|^2+\left(\frac{\kappa_n}{2}+C\kappa_n^2\right)\left|\boldsymbol{E}^{\psi^n}+\boldsymbol{\nu}_{N,N_T}^n\right|^2.
		\end{align}
It follows from the construction of $\boldsymbol{E}^{\psi^n}$ (recalling \eqref{J_gradient_error}), Lemma~\ref{lem:X_gradient_approx} and H\"{o}lder's inequality that, for any $n\geq1$,
		\begin{align*}
			\left|\boldsymbol{E}^{\psi^n}\right|&\leq \left\|h^{\theta^*(\psi^n)}(\cdot,X_{\cdot}^{\alpha^{\psi^n}})\right\|_{L_2}\left\|\nabla_x u^{\psi^n}(\cdot,X_{\cdot}^{\alpha^{\psi^n}})\right\|_{L^2}\left\|F^{\psi^n,\delta}\right\|_{L^2}+\left\|\nabla_{\psi}\alpha^{\psi^n}\right\|_{L^2}\left\|I^{\psi^n}\right\|_{L^2}\\
            &\leq C\left(\delta_n+\sqrt{L(\theta^*({\psi^n});\psi^n)}\right),
		\end{align*}
		where, $C>0$ is a constant only depending on $T$ and $M$. Hence, we can deduce from Cauchy-Schwartz's inequality that 
		\begin{align*}
			\mathcal{J}(\psi^{n+1})-\mathcal{J}(\psi^n)&\leq -\left(\frac{\kappa_n}{2}-C\kappa_n^2\right)\left|\nabla_{\psi}\mathcal{J}(\psi^n)\right|^2+\left(\kappa_n+2C\kappa_n^2\right)\left|\boldsymbol{\nu}_{N,N_T}^n\right|^2\nonumber\\
			&\quad+C\kappa_n\left(\delta_n^2+L(\theta^*({\psi^n});\psi^n)\right).
		\end{align*}
		Summing both sides from $0$ to $n$, and applying \eqref{outer_condition}, we can conclude that by using a similar argument used in the proof of Proposition~\ref{prop:inner_loop} that
		\begin{align*}
			\sum_{n=0}^{\infty}\left(\frac{\kappa_n}{2}-C\kappa_n^2\right)\left|\nabla_{\psi}\mathcal{J}(\psi^n)\right|^2<\infty.
		\end{align*}
		WLOG, we assume that $\kappa_n<1/(2C)$ for all $n\geq1$. Similarly, by utilizing the Lipschitz continuity of $\psi\to\nabla_{\psi}\mathcal{J}(\psi)$, there exists a constant $C>0$ only depending on $T$ and $M$ such that
		\begin{align*}
			\left|\nabla_{\psi}\mathcal{J}(\psi^{n+1})-\nabla_{\psi}\mathcal{J}(\psi^n)\right|\leq C\kappa_n,\quad\forall n\geq1.
		\end{align*}
		We can conclude by taking advantage of Lemma~\ref{lem:auxi_lemma} that $\lim_{n\to\infty}|\nabla_{\psi}\mathcal{J}(\psi^n)|=0$. By applying Taylor's expansion again, we derive that
		\begin{align*}
			\mathcal{J}(\psi^{n+1})-\mathcal{J}(\psi^n)&\geq\left(\nabla_{\psi}\mathcal{J}(\psi^n)\right)^{\T}\left(\psi^{n+1}-\psi^n\right)-\frac{C}{2}\left|\psi^{n+1}-\psi^n\right|^2\\
			&\geq-\frac{7\kappa_n}{4}\left|\nabla_{\psi}\mathcal{J}(\psi^n)\right|^2-\frac{7\kappa_n}{2}\left|\boldsymbol{\nu}_{N,N_T}^n\right|^2-C\kappa_n\left(\delta_n^2+L(\theta^*(\psi^n);\psi^n)\right),
		\end{align*}
		which together with \eqref{J_bound} implies that
		\begin{align*}
			\left|\mathcal{J}(\psi^{n+1})-\mathcal{J}(\psi^n)\right|\leq\frac{7\kappa_n}{4}\left|\nabla_{\psi}\mathcal{J}(\psi^n)\right|^2+\frac{7\kappa_n}{2}\left|\boldsymbol{\nu}_{N,N_T}^n\right|^2+C\kappa_n\left(\delta_n^2+L(\theta^*(\psi^n);\psi^n)\right).
		\end{align*}
		Here, we used the condition $\kappa_n<1/(2C)$. By a similar argument used in the proof of Proposition~\ref{prop:inner_loop}, we can show that $\lim\limits_{n\to\infty}\mathcal{J}(\psi^n)=\mathcal{J}^*$ exists, and moreover, it holds that
		\begin{align*}
			\left|\mathcal{J}(\psi^n)-\mathcal{J}^*\right|\leq \sum_{\lambda=n}^{\infty}\left(\frac{7\kappa_{\lambda}}{4}\left|\nabla_{\psi}\mathcal{J}(\psi^{\lambda})\right|^2+\frac{7\kappa_{\lambda}}{2}\left|\boldsymbol{\nu}_{N,N_T}^{\lambda}\right|^2+C\kappa_{\lambda}\left(\delta_{\lambda}^2+L(\theta^*(\psi^{\lambda});\psi^{\lambda})\right)\right).
		\end{align*}
		This yields the desired error decomposition \eqref{J_error} by using the triangle inequality and the proof is thus complete.
	\end{proof}

\section{Proofs of Auxiliary Results}\label{sec:proof_auxi}

This section collects the proofs of auxiliary results in the paper and we will assume $C>0$ is a generic constant only depending on $T$ and $M$ introduced in Assumption~\ref{ass1} and may vary from line to line.
\begin{proof}[Proof of Lemma~\ref{lem:X_moment}]
For any $\psi\in\R^p$, let us define $b^{\psi}(t,x):=b(t,x,u^{\psi}(t,x))$ and $\sigma^{\psi}(t,x):=\sigma(t,x,u^{\psi}(t,x))$ for $(t,x)\in [0,T]\times\R^d$. In light of Assumption~\ref{ass1}, it holds that, for any $x_1,x_2\in\R^d$,
\begin{align}\label{Lipschitz_condition}
\left|(b^{\psi},\sigma^{\psi})(t,x_1)-(b^{\psi},\sigma^{\psi})(t,x_2)\right|&\leq 2M\left(|x_1-x_2|+|u^{\psi}(t,x_1)-u^{\psi}(t,x_2)|\right)\nonumber\\
&\leq 2(M+M_1(t))|x_1-x_2|,\quad \forall t\in[0,T].
\end{align}
This also implies the linear growth of $(b^{\psi},\sigma^{\psi})$ with respect to $x\in\R^d$. By applying It\^{o}'s rule to $|X_t^{\alpha^{\psi}}|^4=(|X_t^{\alpha^{\psi}}|^2)^2$, we obtain
\begin{align*}
|X_t^{\alpha^{\psi}}|^4&=|\xi|^4+4\int_0^t\left|X_s^{\alpha^{\psi}}\right|^2\left((X_s^{\alpha^{\psi}})^{\T}b^{\psi}(s,X_s^{\alpha^{\psi}})+|\sigma^{\psi}(s,X_s^{\alpha^{\psi}})|^2\right)\d s\\
&\quad+4\int_0^t|X_s^{\alpha^{\psi}}|^2(X_s^{\alpha^{\psi}})^{\T}\sigma^{\psi}(s,X_s^{\alpha^{\psi}})\d W_s.
\end{align*} 
Applying the Young's ineqaulity and BDG's inequality, we have, for any $t\in[0,T]$,
\begin{align}\label{moment_2}
\E\left[\sup_{s\in [0,t]}\left|X_s^{\alpha^{\psi}}\right|^4\right]\leq \E\left[\left|\xi\right|^4\right]+\E\left[\int_0^tC(s)\left|X_s^{\alpha^{\psi}}\right|^4\d s\right],
\end{align}
where, the positive function $t\to C(t)$ only depends on $M$ and satisfies $\int_0^TC(t)\d t<\infty$ (we have used the square-integrability of $t\to M_1(t)$). By using the Gronwall's lemma, the estimate \eqref{moment_2} yields that
\begin{align*}
\E\left[\sup_{t\in [0,T]}\left|X_t^{\alpha^{\psi}}\right|^4\right]\leq \E\left[|\xi|^4\right]\exp\left(\int_0^TC(t)\d t\right)\leq C\left(1+\E\left[|\xi|^4\right]\right).
\end{align*}
Thus, the proof is hence complete.
\end{proof}
	
\begin{proof}[Proof of Lemma~\ref{lem:decoupling_field}]
For any $(t,x)\in [0,T]\times\R^d$, let the triplet $(X^{t,x,\alpha^{\psi}},Y^{t,x,\alpha^{\psi}},Z^{t,x,\alpha^{\psi}})=$\\ $(X^{t,x,\alpha^{\psi}}_s,Y_s^{t,x,\alpha^{\psi}},Z_s^{t,x,\alpha^{\psi}})_{s\in [t,T]}$ be the solution to the following FBSDE starting from $(t,x)$:
{\small\begin{align}\label{FBSDE_tx}
\begin{cases}
\displaystyle X_s^{t,x,\alpha^{\psi}}=x+\int_t^sb(\lambda,X_\lambda^{t,x,\alpha^{\psi}},u^{\psi}(\lambda,X_\lambda^{t,x,\alpha^{\psi}}))\d \lambda+\int_t^s\sigma(\lambda,X_\lambda^{t,x,\alpha^{\psi}},u^{\psi}(\lambda,X_\lambda^{t,x,\alpha^{\psi}}))\d W_\lambda,\\[1em]
\displaystyle Y_s^{t,x,\alpha^{\psi}}=\nabla_xg(X_T^{t,x,\alpha^{\psi}})+\int_s^T\nabla_xH(\lambda,X_\lambda^{t,x,\alpha^{\psi}},u^{\psi}(\lambda,X_\lambda^{t,x,\alpha^{\psi}}),Y_\lambda^{t,x,\alpha^{\psi}},Z_\lambda^{t,x,\alpha^{\psi}})\d \lambda\\[1em]
\displaystyle\qquad\qquad-\int_s^TZ_\lambda^{t,x,\alpha^{\psi}}\d W_\lambda.
\end{cases} 
\end{align}}By using Assumption~\ref{ass1} and Theorem 6.3.3 of Pham \cite{Pham}, there exists a jointly continuous (deterministic) mapping $v^{\psi}:[0,T]\times\R^d\to \R^l$ such that $Y_t^{t,x,\alpha^{\psi}}=v^{\psi}(t,x)$ for all $(t,x)\in[0,T]\times\R^d$.
		
We next verify that $v^{\psi}(t,x)$ is continuously differentiable with respect to the space variable $x\in\R^d$ and satisfies the estimate \eqref{gradient_bound}. To start with, we first define the gradient process $\nabla_x X^{t,x,\alpha^{\psi}}=(\nabla_xX_s^{t,x,\alpha^{\psi}})_{s\in [t,T]}$ by
		\begin{align*}
			\nabla_xX_s^{t,x,\alpha^{\psi}}&=I_d+\int_t^s\nabla_xb^{\psi}(\lambda,X_\lambda^{t,x,\alpha^{\psi}})\nabla_xX_\lambda^{t,x,\alpha^{\psi}}\d\lambda+\int_t^s\nabla_x\sigma^{\psi}(\lambda,X_\lambda^{t,x,\alpha^{\psi}})\nabla_xX_\lambda^{t,x,\alpha^{\psi}}\d W_\lambda,
		\end{align*}
where, $I_d$ denotes the identity matrix in $\R^{d\times d}$ and $(b^{\psi},\sigma^{\psi})$ is defined in the proof of Lemma~\ref{lem:X_moment}. Following the standard moment estimation procedure (or just following the proof of Lemma~\ref{lem:X_gradient_approx}), we can deduce that
\begin{align}\label{X_state_gradient}
\lim_{\epsilon\to 0}~\sum_{i=1}^d~\E\left[\sup_{s\in [t,T]}\left|\frac{X_s^{t,x+\epsilon\boldsymbol{e}_i,\alpha^{\psi}}-X_s^{t,x,\alpha^{\psi}}}{\epsilon}-\nabla_xX_s^{t,x,\alpha^{\psi}}\boldsymbol{e}_i\right|^2\right]=0,
\end{align}
where, for $i=1,\ldots,d$, $\boldsymbol{e}_i$ denotes the $i$-th canonical basis vector in $\R^d$. To ease the notation, we set, for $(t,x,y,z)\in [0,T]\times\R^d\times\R^d\times\R^{d\times r}$,
		\begin{align*}
			F^{\psi}(t,x,y,z):=\nabla_xb(t,x,u^{\psi}(t,x))^{\T}y+\nabla_x\sigma(t,x,u^{\psi}(t,x))^{\T}z+\nabla_xf(t,x,u^{\psi}(t,x)).
		\end{align*}
		Here, we adopt the notation given by
		\begin{align*}
			\left(\nabla_x\sigma(t,x,u^{\psi}(t,x))^{\T}z\right)_i=\tr\left(\pa_{x_i}\sigma(t,x,u^{\psi}(t,x))^{\T}z\right),\quad \forall i=1\dots,d.
		\end{align*} 
		We then define the gradient BSDE by, for any $s\in[t,T]$,
		\begin{align}\label{gradient_BSDE}
			\nabla_x Y^{t,x,\alpha^{\psi}}_s&=\nabla_{x}^2g(X_T^{t,x,\alpha^{\psi}})\nabla_xX_T^{t,x,\alpha^{\psi}}-\int_s^T\nabla_x Z_{\lambda}^{t,x,\alpha^{\psi}}\d W_\lambda\nonumber\\
			&\quad+\int_s^T\nabla_x F^{\psi}\left(\lambda,X_\lambda^{t,x,\alpha^{\psi}},Y_\lambda^{t,x,\alpha^{\psi}},Z_\lambda^{t,x,\alpha^{\psi}}\right)\nabla_x X_\lambda^{t,x,\alpha^{\psi}}\d\lambda\nonumber\\
			&\quad+\int_s^T\nabla_xb\left(\lambda,X_\lambda^{t,x,\alpha^{\psi}},u^{\psi}(\lambda,X_\lambda^{t,x,\alpha^{\psi}})\right)^{\T}\nabla_x Y_\lambda^{t,x,\alpha^{\psi}}\d\lambda\nonumber\\	&\quad+\int_s^T\nabla_x\sigma\left(\lambda,X_\lambda^{t,x,\alpha^{\psi}},u^{\psi}(\lambda,X_\lambda^{t,x,\alpha^{\psi}})\right)^{\T}\nabla_x Z_\lambda^{t,x,\alpha^{\psi}}\d\lambda.
		\end{align}
		Applying It\^{o}'s formula to $\left|\frac{Y_s^{t,x+\epsilon\boldsymbol{e}_i,\alpha^{\psi}}
			-Y_s^{t,x,\alpha^{\psi}}}{\epsilon}-\nabla_xY_s^{t,x,\alpha^{\psi}}\boldsymbol{e}_i\right|^2$ for $i=1,\ldots,d$, and then using BDG's inequality and Gronwall's lemma in a similar manner as in the proof of \eqref{X_state_gradient}, we obtain that
		\begin{align*}
			\lim_{\epsilon\to 0}~\sum_{i=1}^d~&\E\left[\sup_{s\in [t,T]}\left|
			\frac{Y_s^{t,x+\epsilon\boldsymbol{e}_i,\alpha^{\psi}}
				-Y_s^{t,x,\alpha^{\psi}}}{\epsilon}-\nabla_xY_s^{t,x,\alpha^{\psi}}\boldsymbol{e}_i\right|^2\right.\nonumber\\	&\qquad+\left.\int_t^T\left|\frac{Z_s^{t,x+\epsilon\boldsymbol{e}_i,\alpha^{\psi}}-Z_s^{t,x,\alpha^{\psi}}}{\epsilon}
			-\nabla_xZ_s^{t,x,\alpha^{\psi}}\boldsymbol{e}_i\right|^2\d s\right]=0.
		\end{align*}
		In particular, it follows that
		\begin{align*}
			0&=\lim_{\epsilon\to 0}~\sum_{i=1}^d~\E\left[	\left|
			\frac{Y_t^{t,x+\epsilon\boldsymbol{e}_i,\alpha^{\psi}}-Y_t^{t,x,\alpha^{\psi}}}{\epsilon}-\nabla_xY_t^{t,x,\alpha^{\psi}}\boldsymbol{e}_i\right|^2\right]\\
			&=\lim_{\epsilon\to 0}~\sum_{i=1}^d~ \left|\frac{v^{\psi}(t,x+\epsilon\boldsymbol{e}_i)-v^{\psi}(t,x)}{\epsilon}-\nabla_x Y_t^{t,x,\alpha^{\psi}}\boldsymbol{e}_i\right|^2.
		\end{align*}
		This implies that $\nabla_x Y^{t,x,\alpha^{\psi}}$ is deterministic and it further holds that $\nabla_xv^{\psi}(t,x)=\nabla_x Y^{t,x,\alpha^{\psi}}_t$ for any $(t,x)\in[0,T]\times\R^d$. Standard moment estimates yield that $\nabla_x Y^{t,x,\alpha^{\psi}}$ is jointly continuous in $(t,x)\in [0,T]\times \R^d$ and hence $v^{\psi}(t,x)$ is continuously differentiable with respect to the space variable $x\in\R^d$. Finally, it remains to verify that $\nabla_xY^{t,x,\alpha^{\psi}}_t$ is uniformly bounded. By applying It\^{o}'s formula to $|\nabla_x X_s^{t,x,\alpha^{\psi}}|^2$ and utilizing \eqref{Lipschitz_condition}, we can deduce that
\begin{align*}
\E\left[\sup_{s\in [t,T]}\left|\nabla_xX_s^{t,x,\alpha^{\psi}}\right|^2\right]\leq M.
\end{align*}
We then apply It\^{o}'s formula to $|\nabla Y_s^{t,x,\alpha^{\psi}}|^2$ to derive that
\begin{align}\label{gradient_norm}
\left|\nabla_xY_s^{t,x,\alpha^\psi}\right|^2&+\int_s^T
\left|\nabla_xZ_\lambda^{t,x,\alpha^\psi}\right|^2\d\lambda=\left|\nabla_x^2g\left(X_T^{t,x,\alpha^\psi}\right)
			\nabla_xX_T^{t,x,\alpha^\psi}\right|^2\nonumber\\
			&\quad+2\int_s^T\tr\left(
			\left(\nabla_xY_\lambda^{t,x,\alpha^\psi}\right)^{\T}
			\nabla_xF^\psi
			\left(\lambda,X_\lambda^{t,x,\alpha^\psi},
			Y_\lambda^{t,x,\alpha^\psi},
			Z_\lambda^{t,x,\alpha^\psi}\right)
			\nabla_xX_\lambda^{t,x,\alpha^\psi}\right)\d\lambda\nonumber\\
			&\quad+2\int_s^T\tr\left(\left(\nabla_xY_\lambda^{t,x,\alpha^\psi}\right)^{\T}
			\nabla_xb
			\left(\lambda,X_\lambda^{t,x,\alpha^\psi},
			u^\psi(\lambda,X_\lambda^{t,x,\alpha^\psi})\right)^{\T}\nabla_xY_\lambda^{t,x,\alpha^\psi}\right)\d\lambda\nonumber\\
			&\quad+2\int_s^T\tr\left(
			\left(\nabla_xY_\lambda^{t,x,\alpha^\psi}\right)^{\T}\nabla_x\sigma\left(\lambda,X_\lambda^{t,x,\alpha^\psi},u^\psi(\lambda,X_\lambda^{t,x,\alpha^\psi})\right)^{\T}\nabla_xZ_\lambda^{t,x,\alpha^\psi}\right)\d\lambda\nonumber\\
			&\quad-2\int_s^T\tr\left(
			\left(\nabla_xY_\lambda^{t,x,\alpha^\psi}\right)^{\T}
			\nabla_xZ_\lambda^{t,x,\alpha^\psi}\d W_\lambda\right).
		\end{align}
By using Assumption~\ref{ass1}, we have, for any $(t,x,y,z)\in [0,T]\times\R^d\times\R^d\times\R^{d\times r}$, 
\begin{align*}
\left|\nabla_x F^{\psi}(t,x,y,z)\right|&\leq \left(\nabla_{x}^2b(t,x,u^{\psi}(t,x))+\nabla_{xa}^2b(t,x,u^{\psi}(t,x))\nabla_xu^{\psi}(t,x)\right)^{\T}y\\
&\quad+\left(\nabla_{x}^2\sigma(t,x,u^{\psi}(t,x))+\nabla_{xa}^2\sigma(t,x,u^{\psi}(t,x))\nabla_xu^{\psi}(t,x)\right)^{\T}z\\
&\quad+\left(\nabla_{x}^2f(t,x,u^{\psi}(t,x))+\nabla_{xa}^2f(t,x,u^{\psi}(t,x))\nabla_xu^{\psi}(t,x)\right)\\
&\leq C(t)\left(1+|y|+|z|\right),
\end{align*}
where, $t\to C(t)>0$ satisfies $C:=\int_0^TC(t)^2\d t<\infty$ which depends on $T$ and $M$ introduced in  Assumption~\ref{ass1}. Hence,  by taking expectations on both sides of \eqref{gradient_norm}, we conclude from Gronwall's lemma that
\begin{align*}
\E\left[\left|\nabla_xY_t^{t,x,\alpha^\psi}\right|^2+\int_t^T
\left|\nabla_xZ_\lambda^{t,x,\alpha^\psi}\right|^2\d\lambda\right]\leq C(1+|x|^2)
\end{align*}
Thus, we have that $\nabla_x v^{\psi}(t,x)\leq C(1+|x|)$. 

By the pathwise uniqueness of the forward SDE and the BSDE, we can conclude that $Y_t^{\alpha^{\psi}}=v^{\psi}(t,X_t^{\alpha^{\psi}})$ for $t\in[0,T]$,  ${\tt m}\otimes \Pb$-a.s.. On the other hand, in light of the Clark-Ocone's formula (c.f. Proposition 2.1 in Ocone and Karatzas~\cite{OconeKaratzas1991}), we derive,  $m\otimes\Pb$-$\as$\begin{align*}
 Z_t^{\alpha^{\psi}}=D_t Y_t^{\alpha^{\psi}}=\nabla_xv^{\psi}(t,X_t^{\alpha^{\psi}})D_tX_t^{\alpha^{\psi}}=\nabla_xv^{\psi}(t,X_t^{\alpha^{\psi}})\sigma(t,X_t^{\alpha^{\psi}},u^{\psi}(t,X_t^{\alpha^{\psi}}))   
\end{align*}
 where, $D_t$ stands for the Malliavin's derivative operator at time $t\in [0,T]$.  So far, we finished the proof.
\end{proof}
	
\begin{proof}[Proof of Lemma~\ref{lem:X_gradient_approx}]
Recall the mapping $(b^{\psi},\sigma^{\psi})$ defined in the proof of Lemma~\ref{lem:X_moment}. Then, for any $t\in[0,T],
$\begin{align*}
X_t^{\alpha^{\psi+\delta\boldsymbol{v}_m}}-X_t^{\alpha^{\psi}}&=\int_0^t\left(b^{\psi+\delta\boldsymbol{v}_m}(s,X_s^{\alpha^{\psi+\delta\boldsymbol{v}_m}})-b^{\psi}(s,X_s^{\alpha^{\psi}})\right)\d s\\
&\quad+\int_0^t\left(\sigma^{\psi+\delta\boldsymbol{v}_m}(s,X_s^{\alpha^{\psi+\delta\boldsymbol{v}_m}})-\sigma^{\psi}(s,X_s^{\alpha^{\psi}})\right)\d W_s.
\end{align*}
It follows from Assumption~\ref{ass1} that, for any $\psi^1,\psi^2\in\R^p$,
\begin{align}\label{psi_Lipschitz}
\left|(b^{\psi^1},\sigma^{\psi^2})(t,x)-(b^{\psi^2},\sigma^{\psi^2})(t,x)\right|\leq 2M|u^{\psi^1}(t,x)-u^{\psi^2}(t,x)|\leq 2M^2(1+|x|)|\psi^1-\psi^2|.
\end{align}
As a consequence, by applying It\^{o}'s formula to $|	X_t^{\alpha^{\psi+\delta\boldsymbol{v}_m}}-X_t^{\alpha^{\psi}}|^4$ and using Burkholder-Davis-Gundy inequality, Gronwall's lemma and Lemma~\ref{lem:X_moment}, there exists a constant $C>0$ depending on $T$ and $M$ introduced in Assumption \ref{ass1} such that
\begin{align}\label{X_bound}
\E\left[\sup_{t\in [0,T]}	\left|X_t^{\alpha^{\psi+\delta\boldsymbol{v}_m}}-X_t^{\alpha^{\psi}}\right|^4\right]\leq C\delta^4.
\end{align}
This yields that, for any $m=1,\dots,p$,
\begin{align}\label{Delta_bound}
\sup_{\delta>0}\E\left[\sup_{t\in [0,T]}\left|\Delta_t^{\psi,\delta,m}\right|^4\right]\leq C.
\end{align}
For any $m=1,\dots,p$, to ease the notation, we write $(b^{\psi+\delta\boldsymbol{v}_m}(t,X_t^{\alpha^{\psi+\delta\boldsymbol{v}_m}}),\sigma^{\psi+\delta\boldsymbol{v}_m}(t,X_t^{\alpha^{\psi+\delta\boldsymbol{v}_m}})$ $b^{\psi}(t,X_t^{\alpha^{\psi}}),\sigma^{\psi}(t,X_t^{\alpha^{\psi}}),\nabla_xb^{\psi}(t,X_t^{\alpha^{\psi}}),\nabla_x\sigma^{\psi}(t,X_t^{\alpha^{\psi}}),\nabla_{\psi}b^{\psi}(t,X_t^{\alpha^{\psi}}),\nabla_{\psi}\sigma^{\psi}(t,X_t^{\alpha^{\psi}}))$ as $(b_t^{\psi+\delta\boldsymbol{v}_m},\\ \sigma_t^{\psi+\delta\boldsymbol{v}_m},b_t^{\psi},\sigma_t^{\psi},\nabla_xb_t^{\psi},\nabla_x\sigma_t^{\psi},\nabla_{\psi}b_t^{\psi},\nabla_{\psi}\sigma_t^{\psi})$ for short, and it then follows that
\begin{align}\label{X_dif}
&\left|\Delta_t^{\psi,\delta,m}-\pa_{\psi_m}X_t^{\alpha^{\psi}}\right|^2=2\int_0^t\left(\Delta_s^{\psi,\delta,m}-\pa_{\psi_m}X_s^{\alpha^{\psi}}\right)^{\T}\nabla_xb^{\psi}_s\left(\Delta_s^{\psi,\delta,m}-\pa_{\psi_m}X_s^{\alpha^{\psi}}\right)\d s\nonumber\\
&\qquad+2\int_0^t\left(\Delta_s^{\psi,\delta,m}-\pa_{\psi_m}X_s^{\alpha^{\psi}}\right)^{\T}\left(\frac{b^{\psi+\delta\boldsymbol{v}_m}_s-b_s^{\psi}}{\delta}-\nabla_xb^{\psi}_s\Delta_s^{\psi,\delta,m}-\nabla_{\psi}b^{\psi}_s\right)\d s\nonumber\\
&\qquad+\int_0^t\left|\frac{\sigma_s^{\psi+\delta\boldsymbol{v}_m}-\sigma_s^{\psi}}{\delta}-\nabla_x\sigma_s^{\psi}\pa_{\psi_m}X_s^{\alpha^{\psi}}-\nabla_{\psi}\sigma_s^{\psi}\right|^2\d s\nonumber\\
&\qquad+2\int_0^t\left(\Delta_s^{\psi,\delta,m}-\pa_{\psi_m}X_s^{\alpha^{\psi}}\right)^{\T}\nabla_x\sigma_s^{\psi}\left(\Delta_s^{\psi,\delta,m}-\pa_{\psi_m}X_s^{\alpha^{\psi}}\right)\d W_s\nonumber\\
&\qquad+2\int_0^t\left(\Delta_s^{\psi,\delta,m}-\pa_{\psi_m}X_s^{\alpha^{\psi}}\right)^{\T}\left(\frac{\sigma^{\psi+\delta\boldsymbol{v}_m}_s-\sigma_s^{\psi}}{\delta}-\nabla_x\sigma^{\psi}_s\Delta_s^{\psi,\delta,m}-\nabla_{\psi}\sigma^{\psi}_s\right)\d W_s\nonumber\\
&\quad=:I_1(t)+I_2(t)+I_3(t)+I_4(t)+I_5(t).
\end{align}

In the sequel, let $t\to C(t)$ be a deterministic (positive) function satisfying $C:=\int_0^T C(t)^2\d t<\infty$ which depends only on $T$ and $M$ introduced in Assumption \ref{ass1}. The function $t\to C(t)$ may vary from line to line. By the Lipschitz condition established in \eqref{Lipschitz_condition}, one has $\sup_{t\in [0,t]}I_1(s)\leq 2\int_0^tC(s)|\Delta_s^{\psi,\delta,m}-\pa_{\psi_m}X_s^{\alpha^{\psi}}|^2\d s$ for $t\in[0,T]$. Since $b^{\psi}(t,x)$ is twice continuously differentiable in $(\psi,x)\in\R^p\times\R^d$, it follows from  Cauchy-Schwartz inequality that, for any $t\in[0,T]$,
\begin{align*}
\sup_{s\in [0,t]}I_2(t)&\leq 2\sup_{s\in [0,t]}\left|\Delta_s^{\psi,\delta,m}-\pa_{\psi_m}X_s^{\alpha^{\psi}}\right|\int_0^t\left|\frac{b^{\psi+\delta\boldsymbol{v}_m}_s-b_s^{\psi}}{\delta}-\nabla_xb^{\psi}_s\Delta_s^{\psi,\delta,m}-\nabla_{\psi}b^{\psi}_s\right|\d s\\
&\leq\epsilon\sup_{s\in [0,t]}\left|\Delta_s^{\psi,\delta,m}-\pa_{\psi_m}X_s^{\alpha^{\psi}}\right|^2+\frac{\delta^2}{\epsilon}\int_0^tC(s)\left(1+|\Delta_s^{\psi,\delta,m}|^4\right)\d s.
\end{align*}
Here, $\epsilon>0$ is a constant, and in the 2nd inequality, we applied the Taylor's expansion to $b^{\psi}(t,x)$ and utilized the boundedness of its second-order derivatives. For the term $I_3$, we have similarly from Cauchy-Schwartz inequality that, for $t\in[0,T]$,
\begin{align*}
\sup_{s\in [0,t]}I_3(t)\leq 2\int_0^tC(s)\sup_{\lambda\in [0,s]}\left|\Delta_\lambda^{\psi,\delta,m}-\pa_{\psi_m}X_\lambda^{\alpha^{\psi}}\right|^2\d s +2\delta^2\int_0^tC(s)(1+|\Delta_s^{\psi,\delta,m}|^4)\d s.
\end{align*}
By the BDG's inequality and Cauchy-Schwartz inequality, for $t\in[0,T]$,
\begin{align*}
\mathbb E\left[\sup_{s\in[0,t]}I_4(s)\right] &C\leq \mathbb E\left[\left(\int_0^t\left|\Delta_s^{\psi,\delta,m}-\pa_{\psi_m}X_s^{\alpha^{\psi}}\right|^4|\nabla_x\sigma_s^\psi|^2\d s\right)^{1/2}\right]\\
&\leq C\E\left[\left(\sup_{s\in [0,t]}	\left|\Delta_s^{\psi,\delta,m}-\pa_{\psi_m}X_s^{\alpha^{\psi}}\right|\int_0^t
\left|\Delta_s^{\psi,\delta,m}-\pa_{\psi_m}X_s^{\alpha^{\psi}}\right|^2|\nabla_x\sigma_s^\psi|^2\d s\right)^{\frac12}\right]\\
&\leq\epsilon\E\left[\sup_{s\in [0,t]}	\left|\Delta_s^{\psi,\delta,m}-\pa_{\psi_m}X_s^{\alpha^{\psi}}\right|^2\right]+\frac{C^2}{4\epsilon}\E\left[\int_0^t\left|\Delta_s^{\psi,\delta,m}-\pa_{\psi_m}X_s^{\alpha^{\psi}}\right|^2|\nabla_x\sigma_s^\psi|^2\d s\right]\\
&\leq \epsilon\E\left[\sup_{s\in [0,t]}	\left|\Delta_s^{\psi,\delta,m}-\pa_{\psi_m}X_s^{\alpha^{\psi}}\right|^2\right]+\frac{1}{4\epsilon}\E\left[
\int_0^t C(s)\sup_{\lambda\in [0,s]}\left|\Delta_\lambda^{\psi,\delta,m}-\pa_{\psi_m}X_\lambda^{\alpha^{\psi}}\right|^2\d s\right].
\end{align*}
Again by applying the BDG's inequality, Cauchy-Schwartz inequality  and Taylor's expansion, it follows that, for $t\in[0,T]$, 
\begin{align*}
&\mathbb E\left[\sup_{s\in[0,t]}I_5(s)\right]\nonumber\\
&\qquad\leq C\mathbb E\left[
\sup_{s\in[0,t]}\left|\Delta_s^{\psi,\delta,m}-\pa_{\psi_m}X_s^{\alpha^{\psi}}\right|\left(\int_0^t\left|\frac{\sigma^{\psi+\delta\boldsymbol{v}_m}_s-\sigma_s^{\psi}}{\delta}-\nabla_x\sigma^{\psi}_s\Delta_s^{\psi,\delta,m}-\nabla_{\psi}\sigma^{\psi}_s\right|^2\d s\right)^{\frac12}\right]\\
&\qquad\leq\epsilon\mathbb E\left[\sup_{s\in[0,t]}
\left|\Delta_s^{\psi,\delta,m}-\pa_{\psi_m}X_s^{\alpha^{\psi}}\right|^2\right]+\frac{\delta^2}{4\epsilon}\mathbb E\left[\int_0^TC(s)\left(1+|\Delta_s^{\psi,\delta,m}|^4\right)\d s\right].
\end{align*}
Combing the above estimations altogether and choosing $\epsilon=\frac14$, we can finally conclude from Gronwall's lemma that
\begin{align*}
\E\left[\sup_{t\in [0,T]}\left|\Delta_t^{\psi,\delta,m}-\pa_{\psi_m}X_t^{\alpha^{\psi}}\right|^2\right]\leq C\delta^2.
\end{align*} 
Since $m\in\{1,\dots,p\}$ is arbitrary, the proof is hence complete.
\end{proof}

\begin{proof}[Proof of Lemma \ref{lem:equivalence-dpg}]
		Let us introduce the value-function-based adjoint processes by $\bar Y_t^{\alpha^\psi}:=\nabla_xV^\psi(t,X_t^{\alpha^\psi})$ and $\bar Z_t^{\alpha^\psi}:=\nabla_{x}^2V^\psi(t,X_t^{\alpha^\psi})
		\sigma(t,X_t^{\alpha^\psi},u^\psi(t,X_t^{\alpha^\psi}))$ for $t\in[0,T]$.
		Differentiating \eqref{eq:policy-evaluation}  and comparing with \eqref{Hamiltonian} yields that
		\begin{align} \label{eq:Aa-Ha}	\nabla_aA^\psi(t,X_t^{\alpha^\psi},u^\psi(t,X_t^{\alpha^\psi}))	=\nabla_aH(t,X_t^{\alpha^\psi},u^\psi(t,X_t^{\alpha^\psi}),\bar Y_t^{\alpha^\psi},\bar Z_t^{\alpha^\psi}).
		\end{align}

Next, we compare $(\bar Y^{\alpha^\psi},\bar Z^{\alpha^\psi})=(\bar Y_t^{\alpha^\psi},\bar Z_t^{\alpha^\psi})_{t\in[0,T]}$ with the standard SMP adjoint process $(Y^{\alpha^\psi},Z^{\alpha^\psi})=(Y_t^{\alpha^\psi},Z_t^{\alpha^\psi})_{t\in[0,T]}$. Differentiating Eq.~\eqref{eq:policy-evaluation} with respect to $x\in\R^d$, and applying It\^o's rule to $\nabla_xV^\psi(t,X_t^{\alpha^\psi})$ yields that
		\begin{align}\label{eq:closed-loop-adjoint-bsde}
			d\bar Y_t^{\alpha^\psi}&=-\left[ \nabla_xH(t,X_t^{\alpha^\psi},\alpha^\psi_t,\bar Y_t^{\alpha^\psi},\bar Z_t^{\alpha^\psi})+\nabla_xu^\psi(t,X_t^{\alpha^\psi})^\top
			\nabla_aH(t,X_t^{\alpha^\psi},\alpha^\psi_t,\bar Y_t^{\alpha^\psi},\bar Z_t^{\alpha^\psi})\right]\d t\nonumber\\
			&\quad+ \bar Z_t^{\alpha^\psi}\d W_t,
		\end{align}
		where the terminal condition $\bar Y_T^{\alpha^\psi}=\nabla_xg(X_T^{\alpha^\psi})$ and $\alpha_t^{\psi}=u^\psi(t,X_t^{\alpha^\psi})$ for $t\in[0,T]$. Define $\Delta Y_t:=\bar Y_t^{\alpha^\psi}-Y_t^{\alpha^\psi}$ and $\Delta Z_t:=\bar Z_t^{\alpha^\psi}-Z_t^{\alpha^\psi}$ for $t\in[0,T]$. Since $\bar Y_T^{\alpha^\psi}=Y_T^{\alpha^\psi}$, we have $\Delta Y_T=0$. Fix an arbitrary direction $v\in\mathbb R^p$ and set $S_t^v:= \nabla_\psi X_t^{\alpha^\psi}v$ for $t\in[0,T]$. Subtracting \eqref{eq:BSDE} from
		\eqref{eq:closed-loop-adjoint-bsde} and applying It\^o's rule to $\Delta Y_t^\top \nabla_{\psi}X_t^{\alpha^{\psi}}$, we have
		\begin{align}\label{eq:duality-identity}
			0&=\E\left[\Delta Y_T^\top S_T^v-\Delta Y_0^\top S_0^v\right]\nonumber\\
			&=\E\bigg[ \int_0^T(\nabla_\psi u^\psi(t,X_t^{\alpha^\psi})v)^\top
			\left(\nabla_aH(t,X_t^{\alpha^\psi},\alpha^\psi_t,\bar Y_t^{\alpha^\psi},\bar Z_t^{\alpha^\psi})-\nabla_aH(t,X_t^{\alpha^\psi},\alpha^\psi_t,Y_t^{\alpha^\psi},Z_t^{\alpha^\psi})\right)\nonumber\\
			&\qquad\quad-\left( \nabla_xu^\psi(t,X_t^{\alpha^\psi})S_t^v\right)^\top \nabla_aH(t,X_t^{\alpha^\psi},\alpha^\psi_t,Y_t^{\alpha^\psi},Z_t^{\alpha^\psi}) \Big)\d t\bigg].
		\end{align}
		Consequently, one has
		\begin{align}\label{eq:key-duality}
			&\E\left[\int_0^T \left(\nabla_xu^\psi(t,X_t^{\alpha^\psi})S_t^v
			+\nabla_\psi u^\psi(t,X_t^{\alpha^\psi})v\right)^\top \nabla_aH\left(t,X_t^{\alpha^\psi},\alpha^\psi_t,Y_t^{\alpha^\psi},Z_t^{\alpha^\psi})\right)\d t \right]\nonumber\\
			&\qquad=
			\E\left[\int_0^T(\nabla_\psi u^\psi(t,X_t^{\alpha^\psi})v)^\top
			\nabla_aH\left(t,X_t^{\alpha^\psi},\alpha^\psi_t,\bar Y_t^{\alpha^\psi},\bar Z_t^{\alpha^\psi}\right)\d t\right].
		\end{align}
		We then have from Using \eqref{eq:Aa-Ha} and \eqref{eq:key-duality} that
		\begin{align*}
			&v^\top\E\left[\int_0^T\left(\nabla_{\psi}u^{\psi}(t,X_t^{\alpha^{\psi}})+\nabla_xu^{\psi}(t,X_t^{\alpha^{\psi}})\nabla_{\psi}X_t^{\alpha^{\psi}}\right)^{\T}\nabla_aH(t,X_t^{\alpha^{\psi}},\alpha_t^{\psi},Y_t^{\alpha^{\psi}},Z_t^{\alpha^{\psi}})\d t\right]\nonumber\\
			&\quad=\E\left[\int_0^T (\nabla_xu^\psi(t,X_t^{\alpha^\psi})S_t^v
			+\nabla_\psi u^\psi(t,X_t^{\alpha^\psi})v)^\top \nabla_aH(t,X_t^{\alpha^\psi},\alpha^\psi_t,Y_t^{\alpha^\psi},Z_t^{\alpha^\psi}))\d t \right]\nonumber\\
			&\quad=\E\left[ \int_0^T(\nabla_\psi u^\psi(t,X_t^{\alpha^\psi})v)^\top
			\nabla_aH(t,X_t^{\alpha^\psi},\alpha^\psi_t,\bar Y_t^{\alpha^\psi},\bar Z_t^{\alpha^\psi})\d t\right]\nonumber\\
			&\quad=\E\left[ \int_0^T(\nabla_\psi u^\psi(t,X_t^{\alpha^\psi})v)^\top\nabla_a A^\psi(t,X_t^{\alpha^\psi},u^\psi(t,X_t^{\alpha^\psi}))\d t \right].
		\end{align*}
		Since $v\in\mathbb R^p$ is arbitrary, the desired result \eqref{eq:dpg-equivalence-main} follows.
	\end{proof}
	
\begin{proof}[Proof of Lemma~\ref{lem:auxi_lemma}]
We argue by contradiction. Suppose that $a_n$ does not converge to $0$ as $n\to\infty$. Then, there exists some $\epsilon>0$ such that $\limsup_{n\to\infty}a_n>\epsilon$. As a result, there exist infinitely many indices $n$ such that $a_n\geq\epsilon$. Since $|a_{n+1}-a_n|\leq Cr_n$, we have, for any $m>n$,
\begin{align*}
|a_m-a_n|&\leq \sum_{i=n}^{m-1}|a_{i+1}-a_i|\leq C\sum_{i=n}^{m-1}r_i.
\end{align*}
We now construct inductively a sequence of indices $\{n_k\}_{k\geq1}$ such that $a_{n_k}\geq\epsilon$ and the corresponding intervals $[n_k,m_k)$ are
disjoint. Given $n_k$, let $m_k>n_k$ be the smallest integer such that
\begin{align*}
\sum_{i=n_k}^{m_k-1}r_i\geq \frac{\epsilon}{2C},
\end{align*}
where, $m_k$ exists since $\sum_{n=1}^{\infty}r_n=\infty$. For every $i\in\{n_k,\ldots,m_k-1\}$, by the minimality of $m_k$, we deduce that
\begin{align*}
\sum_{j=n_k}^{i-1}r_j<\frac{\epsilon}{2C},
\end{align*}
which yields that
\begin{align*}
a_i\geq a_{n_k}-\sum_{j=n_k}^{i-1}|a_{j+1}-a_j|\geq \epsilon-C\sum_{j=n_k}^{i-1}r_j>\frac{\epsilon}{2}.
\end{align*}
Consequently, we have
\begin{align}\label{lower_bound_sum}
\sum_{i=n_k}^{m_k-1}r_i a_i^2\geq\frac{\epsilon^2}{4}\sum_{i=n_k}^{m_k-1}r_i\geq\frac{\epsilon^3}{8C}.
\end{align}
Since there are infinitely many indices for which $a_n\geq\epsilon$, we
can choose $n_{k+1}>m_k$ such that $a_{n_{k+1}}\geq\epsilon$. Thus, the intervals $[n_k,m_k)$ for $k\in\mathbb{N}$ are pairwisely disjoint. It follows from \eqref{lower_bound_sum} that
\begin{align*}
\sum_{n=1}^{\infty}r_na_n^2\geq\sum_{k=1}^{\infty}\sum_{i=n_k}^{m_k-1}r_i a_i^2\geq\sum_{k=1}^{\infty}\frac{\epsilon^3}{8C}=\infty,
\end{align*}
which contradicts the assumption that $\sum_{n=1}^{\infty}r_na_n^2<\infty$. The proof is hence complete.
	\end{proof}

\begin{proof}[Proof of Lemma~\ref{lem:Hastable}]
Let $\alpha,\beta\in\A[0,T]$ be fixed. By applying It\^{o}'s formula to $|X^{\alpha}_t-X^{\beta}|^2$, we have, for $t\in[0,T]$,
\begin{align*}
\left|X_t^{\alpha}-X_t^{\beta}\right|^2&=2\int_0^t (X_s^{\alpha}-X_s^{\beta})^{\T}\left(b(s,X_s^{\alpha},\alpha_s)-b(s,X_s^{\beta},\beta_s)\right)\d s\\
&\quad+\int_0^t\left|\sigma(s,X_s^{\alpha},\alpha_s)-\sigma(s,X_s^{\beta},\beta_s)\right|^2\d s\\
&\quad+2\int_0^t (X_s^{\alpha}-X_s^{\beta})^{\T}\left(\sigma(s,X_s^{\alpha},\alpha_s)-\sigma(s,X_s^{\beta},\beta_s)\right)\d W_s.
\end{align*}
Taking supremum and expectations on both sides of the above display, we have from BDG's inequality and Assumption~\ref{ass1} that
\begin{align*}
\E\left[\sup_{s\in [0,t]}\left|X_s^{\alpha}-X_s^{\beta}\right|^2\right]&\leq C\E\left[\int_0^t|X_s^{\alpha}-X_s^{\beta}|^2\d s+\int_0^t|\alpha_s-\beta_s|^2\d s\right]\\
&\quad+2\left(\int_0^t|X_s^{\alpha}-X_s^{\beta}|^2\left|\sigma(s,X_s^{\alpha},\alpha_s)-\sigma(s,X_s^{\beta},\beta_s)\right|^2\d s\right)^{\frac12}\\
&\leq C\E\left[\int_0^t|X_s^{\alpha}-X_s^{\beta}|^2\d s+\int_0^t|\alpha_s-\beta_s|^2\d s\right]\\
&\quad+\epsilon	\E\left[\sup_{s\in [0,T]}|X_s^{\alpha}-X_s^{\beta}|^2\right]+\frac{C}{\epsilon}\E\left[\int_0^t|X_s^{\alpha}-X_s^{\beta}|^2\d s+\int_0^t|\alpha_s-\beta_s|^2\d s\right],
\end{align*}
where, $\epsilon\in (0,1)$ is arbitrary and we used Cauchy-Schwartz inequality in the last inequality. It then follows from Gronwall's lemma that
\begin{align}\label{X_Lipschitz}
\E\left[\sup_{t\in [0,T]}\left|X_t^{\alpha}-X_t^{\beta}\right|^2\right]\leq C\E\left[\int_0^T|\alpha_t-\beta_t|^2\d t\right].
\end{align}
By applying It\^{o}'s rule to $|Y_t^{\alpha}|^2$, we obtain, for any $t\in[0,T]$,
\begin{align*}
\left|Y_t^{\alpha}\right|^2+\int_t^T\left|Z_s^{\alpha}\right|^2\d s&=|\nabla_xg(X_T^{\alpha})|^2+2\int_t^T(Y_s^{\alpha})^{\T}\nabla_x H(s,X_s^{\alpha},\alpha_s,Y_s^{\alpha},Z_s^{\alpha})\d s\nonumber\\
&\quad+2\int_t^T(Y_s^{\alpha})^{\T}Z_s^{\alpha}\d W_s.
\end{align*}
It follows from Assumption~\ref{ass1} that $|\nabla_xH(t,x,a,y,z)|\leq M(1+|x|+|a|+|y|+|z|)$ for any $(t,x,a,y,z)\in [0,T]\times\R^d\times\R^l\times\R^d\times\R^{d\times r}$. This yields from BDG's inequality, Cauchy-Schwartz inequality and Gronwall's lemma that
\begin{align}\label{Y_bound}
\E\left[\sup_{t\in [0,T]}|Y_t^{\alpha}|^2+\int_0^T|Z_t^{\alpha}|^2\d t\right]\leq C\E\left[\left(1+\int_0^T|\alpha_t|^2\d t\right)\right].
\end{align}
By applying It\^{o}'s rule to $|Y_t^{\alpha}-Y_t^{\beta}|^2$ again, we deduce that
\begin{align*}
|Y_t^{\alpha}-Y_t^{\beta}|^2&=2\int_t^T(Y_s^{\alpha}-Y_s^{\beta})^{\T}\left(\nabla_x H(s,X_s^{\alpha},\alpha_s,Y_s^{\alpha},Z_s^{\alpha})-\nabla_x H(s,X_s^{\beta},\beta_s,Y_s^{\beta},Z_s^{\beta})\right)\d s\\
&\quad-\int_t^T|Z_s^{\alpha}-Z_s^{\beta}|^2\d s+|\nabla_xg(X_T^{\alpha})-\nabla_xg(X_T^{\beta})|^2+2\int_t^T(Y_s^{\alpha}-Y_s^{\beta})^{\T}(Z_s^{\alpha}-Z_s^{\beta})\d W_s.
\end{align*}
By using \eqref{Y_bound} and the local Lipschitz continuity of $(x,a,y,z)\to\nabla_xH(t,x,a,y,z)$ for every $t\in [0,T]$, we conclude from BDG's inequality, Cauchy-Schwartz inequality and Gronwall's lemma that
\begin{align*}
\E\left[\sup_{t\in [0,T]}|Y_t^{\alpha}-Y_t^{\beta}|^2+\int_0^T|Z_t^{\alpha}-Z_t^{\beta}|^2\d t\right]\leq C_{\|\alpha\|_{L^2}\vee \|\beta\|_{L^2}}\E\left[\int_0^T|\alpha_t-\beta_t|^2\d t\right],
\end{align*}
where, $C_{\|\alpha\|_{L^2}^2\vee \|\beta\|_{L^2}^2}>0$ is a constant depending on $T$, $M$ introduced in Assumption \ref{ass1} and $\|\alpha\|_{L^2}^2\vee\|\beta\|_{L^2}^2$. Thanks to Lemma~\ref{lem:X_moment} and Assumption~\ref{ass1}, we have
\begin{align*}
\sup_{\psi\in\R^p}\E\left[\int_0^T|\alpha_t^{\psi}|^2\d t\right]<\infty.
\end{align*}
This implies that, there exists a constant $C>0$ only depending on $T$ and $M$ introduced in Assumption \ref{ass1} such that
\begin{align}\label{DJ_Lipschitz}
\E\left[\sup_{t\in [0,T]}\left|Y_t^{\alpha^{\psi^1}}-Y_t^{\alpha^{\psi^2}}\right|^2+\int_0^T\left|Z_t^{\alpha^{\psi^1}}-Z_t^{\alpha^{\psi^2}}\right|^2\d t\right]\leq C\E\left[\int_0^T\left|\alpha_t^{\psi^1}-\alpha_t^{\psi^2}\right|^2\d t\right].
\end{align} 
The desired result thus follows from \eqref{DJ_Lipschitz} by utilizing Assumption~\ref{ass1}.
\end{proof}
	
\begin{proof}[Proof of Lemma~\ref{lem:gradientstable}] To simplify the notations, for $i=1,2$, we set $G_t^i:=\nabla_{\psi}X_t^{\alpha^{\psi^i}}$, $U_t^i:=\nabla_{\psi}u^{\psi^i}(t,X_t^{\alpha^{\psi^i}})$, $C_t^i:=\nabla_ab(t,X_t^{\alpha^{\psi^i}},\alpha_t^{\psi^i})U_t^i$ and $D_t^i:=\nabla_a\sigma(t,X_t^{\alpha^{\psi^i}},\alpha_t^{\psi^i})U_t^i$ for $t\in[0,T]$. Then, SDE \eqref{X_gradient} can be written as:
\begin{align}\label{gradient_SDE_compact}
dG_t^i=\left(\nabla_xb^{\psi^i}(t,X_t^{\alpha^{\psi^i}})G_t^i+C_t^i\right)\d t+\left(\nabla_x\sigma^{\psi^i}(t,X_t^{\alpha^{\psi^i}})G_t^i+D_t^i\right)\d W_t,\quad G_0^i=\boldsymbol{0}_{d\times p}.
\end{align}
Using Assumption~\ref{ass1}, there exists a deterministic function $t\to C(t)>0$ satisfying $C:=\int_0^TC(t)^2<\infty$ which only depends on $T$ and $M$ introduced in Assumption \ref{ass1} such that, for any $t\in[0,T]$,
\begin{align}\label{gradient_coefficients_bound}
\left|\nabla_xb^{\psi^i}(t,X_t^{\alpha^{\psi^i}})\right|+\left|\nabla_x\sigma^{\psi^i}(t,X_t^{\alpha^{\psi^i}})\right|\leq C(t),\quad\left|C_t^i\right|+\left|D_t^i\right|
\leq C(t)(1+|X_t^{\alpha^{\psi^i}}|),\quad i=1,2.
\end{align}
Furthermore, the standard estimates for the state equation yield that (similar as \eqref{X_bound})
\begin{align}\label{state_stability_L4}
\E\left[\sup_{t\in[0,T]}
\left|X_t^{\alpha^{\psi^1}}-X_t^{\alpha^{\psi^2}}\right|^4\right]\leq C\left|\psi^1-\psi^2\right|^4,
\end{align}
while the standard moment estimate for the linear SDE \eqref{gradient_SDE_compact} together with Lemma~\ref{lem:X_moment} gives that
\begin{align}\label{gradient_L4_bound}
\E\left[\sup_{t\in[0,T]}\left|G_t^i\right|^4\right]\leq C,\quad i=1,2.
\end{align}
We then set $\Delta G_t:=G_t^1-G_t^2$ for $t\in[0,T]$. Subtracting the two equations in \eqref{gradient_SDE_compact}, we obtain
\begin{align}\label{gradient_difference_SDE}
d(\Delta G_t)=&\left[\nabla_xb^{\psi^1}(t,X_t^{\alpha^{\psi^1}})\Delta G_t
			+(\nabla_xb^{\psi^1}(t,X_t^{\alpha^{\psi^1}})-\nabla_xb^{\psi^2}(t,X_t^{\alpha^{\psi^2}}))G_t^2
			+C_t^1-C_t^2\right]\d t\nonumber\\
			&+\left[
			\nabla_x\sigma^{\psi^1}(t,X_t^{\alpha^{\psi^1}})\Delta G_t+(\nabla_x\sigma^{\psi^1}(t,X_t^{\alpha^{\psi^1}})-\nabla_x\sigma^{\psi^2}(t,X_t^{\alpha^{\psi^2}}))G_t^2+D_t^1-D_t^2\right]\d W_t,\nonumber\\
			\Delta G_0=&0.
		\end{align}
By the Lipschitz continuity of the first-order derivatives of $b$ and
$\sigma$ with respect to $(x,a)$, together with the corresponding regularity of $u^\psi$, we have
\begin{align}\label{AB_stability}
&\left|\nabla_xb^{\psi^1}(t,X_t^{\alpha^{\psi^1}})-\nabla_xb^{\psi^2}(t,X_t^{\alpha^{\psi^2}})\right|+\left|\nabla_x\sigma^{\psi^1}(t,X_t^{\alpha^{\psi^1}})-\nabla_x\sigma^{\psi^2}(t,X_t^{\alpha^{\psi^2}})\right|\nonumber\\
&\qquad\qquad\leq C\left[\left|X_t^{\alpha^{\psi^1}}-X_t^{\alpha^{\psi^2}}\right|+\left(1+\left|X_t^{\alpha^{\psi^1}}\right|\right)\left|\psi^1-\psi^2\right|\right],
\end{align}
and
\begin{align}\label{CD_stability}
\left|C_t^1-C_t^2\right|+\left|D_t^1-D_t^2\right|\leq C\left[\left|X_t^{\alpha^{\psi^1}}-X_t^{\alpha^{\psi^2}}\right|+\left(1+\left|X_t^{\alpha^{\psi^1}}\right|\right)\left|\psi^1-\psi^2\right|\right].
\end{align}
From \eqref{gradient_difference_SDE}, the BDG's inequality and
\eqref{gradient_coefficients_bound}, it follows that, for any $t\in[0,T]$,
\begin{align}\label{gradient_difference_estimate}
\E\left[\sup_{s\in[0,t]}\left|\Delta G_s\right|^2\right]&\leq C\int_0^t\E\left[\left|\Delta G_s\right|^2\right]\d s+C\int_0^t\E\left[\left|(A_s^1-A_s^2)G_s^2\right|^2\right]\d s\\
&\quad+C\int_0^t\E\left[\left|(B_s^1-B_s^2)G_s^2\right|^2\right]\d  s+C\int_0^t\E\left[\left|C_s^1-C_s^2\right|^2+\left|D_s^1-D_s^2\right|^2\right]\d s,\nonumber
\end{align}
where, we set $A_t^i:=\nabla_x b^{\psi^i}(t,X_t^{\alpha^{\psi^i}})$ and $B_t^i:=\nabla_x\sigma^{\psi^i}(t,X_t^{\alpha^{\psi^i}})$ with $i=1,2$ for notation simplicity.
The estimate \eqref{AB_stability} and H\"older's inequality yield that, for $s\in[0,T]$,
\begin{align}\label{A_product_estimate}
\E\left[\left|(A_s^1-A_s^2)G_s^2\right|^2\right]&\leq C\E\left[\left(\left|X_s^{\alpha^{\psi^1}}-X_s^{\alpha^{\psi^2}}\right|^2+\left(1+|X_s^{\alpha^{\psi^2}}|\right)^2
\left|\psi^1-\psi^2\right|^2\right)\left|G_s^2\right|^2\right]\nonumber\\
&\leq C\left(\E\left[\left|X_s^{\alpha^{\psi^1}}-X_s^{\alpha^{\psi^2}}\right|^4\right]\right)^{1/2}
\left(\E\left[\left|G_s^2\right|^4\right]\right)^{1/2}\nonumber\\
&\quad+	C\left\{1+\left(\E\left[\left|X_s^{\alpha^{\psi^1}}\right|^4\right]\right)^{\frac12}\left(\E\left[\left|G_s^2\right|^4\right]\right)^{1/2}\right\}\left|\psi^1-\psi^2\right|^2\nonumber\\
&\leq C\left|\psi^1-\psi^2\right|^2,
\end{align}
where, we used \eqref{state_stability_L4}, \eqref{gradient_L4_bound} and \eqref{AB_stability} in the last inequality. Similarly, we have
\begin{align}\label{B_product_estimate}
\E\left[\left|\left(B_s^1-B_s^2\right)G_s^2\right|^2\right]\leq C\left|\psi^1-\psi^2\right|^2.
\end{align}
Furthermore, it follows from \eqref{state_stability_L4}, \eqref{CD_stability} and the proof of \eqref{A_product_estimate} that
\begin{align}\label{CD_product_estimate}
\E\left[\left|C_s^1-C_s^2\right|^2+\left|D_s^1-D_s^2\right|^2\right]\leq C\left|\psi^1-\psi^2\right|^2.
\end{align}
Combining \eqref{gradient_difference_estimate}-\eqref{CD_product_estimate}, we derive from Gronwall's lemma that
\begin{align}\label{gradient_difference_final}
\E\left[\sup_{t\in[0,T]}\left|\Delta G_t\right|^2\right]\leq C\left|\psi^1-\psi^2\right|^2.
\end{align}
Recalling the definition of $\Delta G_t$, the proof is then complete.
	\end{proof}

\end{document}